\documentclass[12pt]{amsart}
\usepackage[margin=1.25in]{geometry}
\usepackage{amsmath, amssymb, amsthm,fancyhdr}
\usepackage[utf8]{inputenc}
\usepackage{enumitem}
\usepackage{tikz-cd}
\usepackage{xfrac}
\usepackage{dsfont}
\usepackage{cancel}
\usepackage{csquotes}
\usepackage{mathrsfs}
\usepackage{calc} 
\usepackage{tensor}
\usepackage[pdfencoding=unicode,pdfusetitle]{hyperref}
\hypersetup{
    colorlinks=true,
    linkcolor=blue,
    filecolor=purple,      
    urlcolor=[rgb]{0 0 .6},
	 psdextra,
}

\usepackage[backend=biber,style=alphabetic,citestyle=alphabetic,url=false,isbn=false,maxnames=99,maxalphanames=99]{biblatex}
\DeclareFieldFormat{doi}{%
  \hfil\penalty90\hfilneg\space DOI\addcolon\addnbspace
  \ifhyperref
    {\href{https://doi.org/#1}{\nolinkurl{#1}}}
    {\nolinkurl{#1}}}

\newcommand{\tetr}[4]{\big\{\begin{smallmatrix}{#1},{#2},{#3}\\{#4} \end{smallmatrix}\big\}} 

\tikzstyle{blank dot}=[fill=black, draw=black, shape=circle]
\tikzstyle{quat}=[fill={rgb,255: red,255; green,128; blue,0}, draw={rgb,255: red,255; green,128; blue,0}, shape=circle]
\tikzstyle{ass.}=[fill=green, draw=none, shape=circle]

\tikzstyle{red line}=[-, draw=red]
\tikzstyle{new edge style 0}=[-, draw=cyan]
\tikzstyle{new edge style 1}=[-, draw=cyan, label=b]
\tikzstyle{Green Box}=[-, fill=green, draw=green]
\tikzstyle{purple line}=[-, draw={rgb,255: red,128; green,0; blue,128}]

\SetLabelAlign{Center}{\hfil#1\hfil}
\SetLabelAlign{CenterWithParen}{\hfil(\makebox[1.0em]{#1})\hfil}

\definecolor{cyan(process)}{rgb}{0.0, 0.6, 1.0}
\definecolor{blue-violet}{rgb}{0.54, 0.17, 0.89}

\newcommand{\id}{\textsf{id}}
\newcommand{\1}{\mathds{1}}

\newcommand{\s}{\mathcal}
\newcommand{\bb}{\mathbb}

\renewcommand{\Re}{\mathsf{Re}}
\newcommand{\Hom}{\mathrm{Hom}}
\newcommand{\End}{\mathrm{End}}

\newcommand{\coev}{\mathrm{coev}}
\newcommand{\im}{\mathrm{im}}

\newcommand{\inv}{^{-1}}

\newcommand{\FPdim}{\mathrm{FPdim}}
\renewcommand{\Vec}{\mathrm{Vec}}
\newcommand{\Rep}{\mathrm{Rep}}

\renewcommand{\lim}{\mathop{\mathrm{lim}}}

\newcommand{\TY}{\mathsf{TY}}

\makeatletter
\newtheorem*{rep@theorem}{\rep@title}
\newcommand{\newreptheorem}[2]{%
\newenvironment{rep#1}[1]{%
 \def\rep@title{#2 \ref{##1}}%
 \begin{rep@theorem}}%
 {\end{rep@theorem}}}
\makeatother

\newtheorem{theorem}{Theorem}[section]
\newreptheorem{theorem}{Theorem}
\newtheorem{proposition}[theorem]{Proposition}
\newtheorem{corollary}[theorem]{Corollary}
\newtheorem{lemma}[theorem]{Lemma}

\theoremstyle{definition}
\newtheorem{definition}[theorem]{Definition}
\newtheorem{example}[theorem]{Example}
\newtheorem{note}[theorem]{Note}
\newtheorem{remark}[theorem]{Remark}

\title{Tambara-Yamagami Categories over the Reals:\\ The Non-Split Case}

\author[J. Plavnik]{Julia Plavnik}
\address{Department of Mathematics, Indiana University}
\vspace{-5mm}
\address{Fachbereich Mathematik, Universit\"at Hamburg}
\email{jplavnik@iu.edu}
\author[S. Sanford]{Sean Sanford}
\address{Department of Mathematics, The Ohio State University}
\email{sanford.189@osu.edu}
\author[D. Sconce]{Dalton Sconce}
\address{Department of Mathematics, Indiana University}
\email{dsconce@iu.edu}

\begin{document}


\begin{abstract}
Tambara and Yamagami investigated a simple
set of fusion rules with only one non-invertible object, and proved under which circumstances those rules could be given a coherent associator.  They also classified all of the resulting fusion categories up to monoidal equivalence. 
    
We consider a generalization of such fusion rules to the setting where simple objects are no longer required to be
 split simple.  Over the real numbers, this means that simple objects are
either real, complex, or quaternionic. 
In this context, we prove a similar categorification result to the one of Tambara and Yamagami. 
 
\end{abstract}

\maketitle


\thispagestyle{fancy}
\fancyhead[R]{\scriptsize Hamburger Beitr\"age zur Mathematik \\ ZMP-HH/23-5}







\section{Introduction}



In the late '90s, Daisuke Tambara and Shigeru Yamagami were studying the Hopf algebras whose categories of representations had the same fusion rules as $\Rep_{\mathbb C}(D_8)$, the category of complex representations of the dihedral group of order 8.  This investigation led them to analyze and completely classify all those fusion categories that have 
fusion rules which are similar to that of $\Rep_{\mathbb C}(D_8)$ in \cite{TAMBARA1998692}. Such categories are now referred to as Tambara-Yamagami categories in honor of their work.  Their classification allowed for arbitrary base fields but they assumed that all the simple objects were \emph{split}, that is, $\End(X)$ is isomorphic to the base field whenever $X$ is simple.

A decade later, Etingof, Nikshych, and Ostrik developed a homotopy theoretic description of extension theory for fusion categories over algebraically closed fields in \cite{etingofFusionCategoriesHomotopy2009}. This  paper made explicit a deep connection between tensor categories and higher groupoids. As an application of their theory, they gave a new shorter proof of Tambara and Yamagami's classification result but with the assumption that the base field is algebraically closed.

The techniques of Etingof, Nikshych, and Ostrik require adjustments in order to be extended to fusion categories over non\textendash algebraically closed fields.
Tambara-Yamagami categories, being extensions of pointed categories by $\mathbb Z/2\mathbb Z$, are natural first examples to help clarify the theory of extensions in this new setting.  When working over such fields, fusion categories often have simple objects which are non-split.  Tambara and Yamagami's theorem doesn't apply to such cases, however some of these non-split categories have fusion rules that are very similar to Tambara-Yamagami fusion rules.

Motivated by these examples, this article proposes a natural non-split generalization of the Tambara and Yamagami fusion rules.  We analyze these new fusion rules for categories over $\mathbb R$ and completely classify all possible fusion categories with these fusion rules.  Our results provide examples of infinite families of non-split fusion categories over $\mathbb R$ that we believe are brand new.



We say that a fusion ring is a \emph{Tambara-Yamagami} ring if it has a basis formed by a group of invertible elements and a single non-invertible element $m$ with the property that the multiplicity of $m$ in $m^2$ is $0$. Below is a summary of our main results regarding categorifications of these fusion rings over $\mathbb R$.

Let $\mathcal C$ be a Tambara-Yamagami category over $\mathbb R$, that is, a fusion category for which its underlying fusion ring is a Tambara-Yamagami fusion ring. In the non-split setting, $\End(\1)$ is either $\mathbb R$ or $\mathbb C$. If $\End(\1)\cong\mathbb R$ then $\End(m) \cong \mathbb H$ or $\mathbb C$ and if $\End(\1)\cong\mathbb C$ then $\End(m) \cong \mathbb C$. We analyze these three possibilities and the classification results are stated below.  For more on the construction of these categories, as well as an explanation of the notation, see the relevant sections.

\begin{reptheorem}{Thm:TYQuaternionic}
    Let $A$ be a finite group, let $\tau=\sfrac{\pm1}{\sqrt{4|A|}}$, and let $\chi:A\times A\to \mathbb R^\times$ be a nongedegerate symmetric bicharacter on $A$.
    
    A triple of such data gives rise to a non-split Tambara-Yamagami category \newline $\s C_{\bb H}(A,\chi,\tau)$, with $\End(\1)\cong\bb R$ and $\End(m)\cong\bb H$.  Furthermore, all equivalence classes of such categories arise in this way.  Two categories $\s C_{\bb H}(A,\chi,\tau)$ and $\s C_{\bb H}(A',\chi',\tau')$ are equivalent if and only if $\tau=\tau'$ and there exists an isomorphism $f:A\to A'$ such that for all $a,b\in A$,
    \[\chi'\big(f(a),f(b)\big)\;=\;\chi(a,b)\,.\]
\end{reptheorem}

\begin{remark}
    The existence of a nondegenerate bicharacter on $A$ that takes values in $\mathbb R^\times$ implies that $A\cong\big(\mathbb Z/2\mathbb Z\big)^n$ for some $n$.
\end{remark}

\begin{reptheorem}{Thm:TYComplex}
    Let $G\cong A\rtimes\mathbb Z/2\mathbb Z$ be a finite generalized dihedral group.  Let $\tau=\sfrac{\pm1}{\sqrt{2|G|}}$, let $(-)^g\in\text{Gal}(\mathbb C/\mathbb R)$, and let $\chi:G\times G\to \mathbb C^\times_*$ be a symmetric bicocycle on $G$ with respect to $(-)^g$, whose restriction $\chi\mid_{A\times A}$ is a nongedegerate bicharacter.
    
    A quadruple of such data gives rise to a non-split Tambara-Yamagami category $\s C_{\bb C}(G,g,\chi,\tau)$, with $\End(\1)\cong\bb R$ and $\End(m)\cong\bb C$.  Furthermore, all equivalence classes of such categories arise in this way.  Two categories $\s C_{\bb C}(G,g,\chi,\tau)$ and \linebreak $\s C_{\bb C}(G',g',\chi',\tau')$ are equivalent if and only if $g=g'$ and there exists 
    the following data:
    \begin{enumerate}[label = \roman*, align=CenterWithParen, labelwidth=1.5em]
        \item an isomorphism $f:G\to G'$
        \item a map $(-)^h:\bb C\to\bb C$, either the identity or complex conjugation,
        \item a scalar $\lambda\in S^1\subset \mathbb C$,
    \end{enumerate}
    subject to the condition that:
    \begin{gather*}
        \chi'\Big(f(a),f(b)\Big)=\frac{\lambda\cdot\lambda^{ab}}{\lambda^a\cdot\lambda^b}\cdot\chi(a,b)^h\hspace{4mm}\text{and}\hspace{4mm}
        \frac{\tau'}{\tau}=\frac{\lambda}{\lambda^g}\;,\hspace{7mm}\text{for all }\,a,b\in G\,.
    \end{gather*}
\end{reptheorem}

\begin{remark}
    The existence of a nondegenerate bicharacter on $A$ that takes values in $\mathbb C^\times$ implies that $A$ must be abelian.
\end{remark}

To understand the case in which $\End(\1)\cong\mathbb C$ and $\End(m) \cong \mathbb C$, and also understand why in this case it is not necessarily fusion over $\mathbb C$, we introduce the notion of Galois nontrivial objects in Definition \ref{Def:Galois(Non)Triv}. Moreover, having such objects induces a faithful $\mathbb Z/2\mathbb Z$-grading on the fusion categories.
\begin{reptheorem}{Thm:GaloisGrading}
    All fusion categories $\mathcal C$ over $\mathbb R$ that contain Galois nontrivial objects necessarily admit a faithful grading by the group $\text{Gal}(\mathbb C/\mathbb R)\cong\mathbb Z/2\mathbb Z$.
\end{reptheorem}

Tambara-Yamagami categories are naturally $\mathbb Z/2\mathbb Z$-graded.  Our analysis shows that when both gradings are present they, must be the same.  There are many such categories, as the following theorem shows.

\begin{reptheorem}{Thm:TYFullComplex}
    Let $A$ be a finite group, and let $\chi:A\times A\to \mathbb C^\times$ be a nondegenerate skew-symmetric bicharacter.  Such a pair $(A,\chi)$ gives rise to a non-split Tambara-Yamagami category $\s C_{\overline{\bb C}}(A,\chi)$, with $\End(X)\cong\bb C$ for every simple object $X$.  Furthermore, all equivalence classes of such categories arise in this way.  Two categories $\s C_{\overline{\bb C}}(A,\chi)$ and $\s C_{\overline{\bb C}}(A',\chi')$ are equivalent if and only there exist isomorphisms:
    \begin{enumerate}[label = \roman*, align=CenterWithParen, labelwidth=1.5em]
        \item an isomorphism $f:A\to A'$, and
        \item $(-)^h:\bb C\to\bb C$ (either the identity or complex conjugation),
    \end{enumerate}
    such that $\chi'\big(f(a),f(b)\big)=\chi(a,b)^h$ for all $a,b\in A$.
\end{reptheorem}  

\vspace{2mm}

This paper is organized as follows.  In Section \ref{sec:prelim}, we provide key definitions we will use, and offer some important examples for context. In Section \ref{Sec:FusionCats/R}, we discuss the properties of fusion categories over $\mathbb R$.  We propose a generalization of Tambara-Yamagami categories in the non-split case in Section \ref{Sub:Non-SplitTYCats}. In Section \ref{Sec:Real/Quaternionic}, we analyze the case with real unit and quaternionic non-invertible simple while in Section \ref{Sec:Real/Complex}, we study the case with real unit and complex non-invertible simple. Finally, in Section \ref{Sec:ComplexGalois}, we discuss the case with all simple objects being complex and having a Galois nontrivial object.
  
\subsection*{Acknowledgments} This work
began with the Research Experiences for Undergraduates program at Indiana University
supported by the NSF grant DMS-1757857. It then evolved into a chapter of Sean Sanford’s Ph.D. thesis from 2022.  The research of J.P. was partially supported by NSF grants DMS-1917319 and DMS-2146392 and by Simons Foundation Award 889000 as part of the Simons Collaboration on Global Categorical Symmetries. J.P. 
performed part of this at the Aspen Center for Physics, which is supported by National Science Foundation grant PHY-1607611. 
J.P. would like to thank the hospitality
and excellent working conditions at the Department of Mathematics at the University of Hamburg, where she has carried out part of this research as a Fellow of the Humboldt Foundation.

\section{Preliminaries}\label{sec:prelim}

We refer the reader to~\cite{etingof2015tensor}
for the basic theory of fusion categories, fusion rings, and for the terminology used throughout this article. We fix a field $\mathbb K$. In most of this article, we will focus on the case in which $\mathbb K = \mathbb R$ but for this section, we do not have any restrictions on the field $\mathbb K$.

\begin{definition}
A fusion category $\mathcal C$ over $\mathbb K$ is a $\mathbb K$-linear finite semisimple rigid monoidal category with simple (monoidal) unit.
\end{definition} 
We will denote by $(\mathcal C, \otimes, \alpha, \1, \ell, r)$ the monoidal structure of the fusion category $\mathcal C$. Here  $\alpha$ is the associativity constraint and $\ell$ and $r$ are the left and right unit constraints for the monoidal unit $\1$. 

\begin{remark}
The Grothendieck group $K_0(\mathcal C)$ associated with the underlying finite semisimple abelian category $\mathcal C$ is the abelian free group with a basis given by the isomorphism classes of simple objects, see \cite[Definition 1.5.8]{etingof2015tensor}. Since $\mathcal C$ is a monoidal category, the Grothendieck group $K_0(\mathcal C)$ inherits a ring structure. Moreover, since the category is rigid, this ring is a \emph{fusion ring} (the definition of a fusion ring can be found in \cite[Definition 3.1.7]{etingof2015tensor} for the split case, and in \cite[Definition 3.1.5]{sanfordThesis} for the non-split case). The interested reader can find more details in \cite[Section 4.5]{etingof2015tensor}.
\end{remark} 

A fusion category $\mathcal C$ can be completely described in terms of its Grothendieck ring and  associativity (and unit) constraints satisfying the pentagon (and triangle) axiom.
A \emph{categorification} of a fusion ring is a fusion category with such fusion ring as its Grothendieck ring. One natural question is whether a given fusion ring admits a categorification. One can further inquire about all the possible categorifications up to tensor equivalence.

\begin{example}
Given a finite group $G$, the group ring $\mathbb Z G$ is a fusion ring. This fusion ring is always categorifiable. A fusion category with $\mathbb Z G$ as its Grothendieck ring is called pointed, see \cite[Defintion 5.11.1]{etingof2015tensor}. 

When the field $\mathbb K$ is algebraically closed, any pointed fusion category is tensor equivalent to a category $\mathbb K\text{-}\Vec_G^\omega$ of finite dimensional $\mathbb K$-vector spaces graded by the group $G$ with the associativity constraint twisted by a 3-cocycle $\omega\in Z^3(G,\mathbb K^\times)$ for some $G$ and $\omega$.  This result is a corollary of a theory long known but unpublished, in Sinh's Ph.D. thesis \cite{sinhGrCats}.  See \cite{baezHigherDimAlgV} for a modern discussion.
If $\mathbb K$ is non-algebraically closed, a similar result holds by \cite{zhuPrimes}. 
\end{example}


 The fusion categories $\Rep_{\mathbb C}(G)$ of finite dimensional complex representations of a finite group $G$ are pointed only when $G$ is abelian.  On the other hand, the fusion rings of $\Rep_{\mathbb C}(D_8)$ are \emph{almost} pointed; they have a unique non-invertible simple object. 
 With the expectation that the fusion rings associated with these categories would be the next simplest after pointed fusion categories, Tambara and Yamagami investigated and fully classified all categorifications of such fusion rules \cite{TAMBARA1998692}.  In the next subsection, we will describe this classification, and the remainder of the article will be devoted to generalizing their result.

\subsection{Tambara-Yamagami Fusion Categories: Split Case}\label{Sub:SplitTYCats}


Let $A$ be a finite group. The \emph{(split) Tambara-Yamagami fusion ring} $\TY(A)$ has a $\mathbb Z$-basis $A\sqcup\{m\}$, $m\notin A$. The product is defined as follows
\[a\cdot b=ab\;\;,\;\;a\cdot m=m=m\cdot a\;\;,\;\;m\cdot m=\sum_{c\in A}c,\]
for $a,b\in A$. The involution of the fusion ring is given by $m^* = m$ and $a^*= a^{-1}$, for $a\in A$. 
A fusion category $\mathcal C$ over $\mathbb K$ is said to be a \emph{split Tambara-Yamagami fusion category} if $K_0(\mathcal C)=\TY(A)$, for some finite group $A$, and $\End(X) \cong \mathbb K$, for every simple object $X$ of $\mathcal C$.

Tambara and Yamagami considered such a fusion ring and determined whether or not $\TY(A)\cong K_0(\mathcal C)$ for some fusion category $\mathcal C$ by solving the pentagon equations  \cite{TAMBARA1998692}. Given a fusion ring $\TY(A)$, they constructed fusion categories $\mathcal C(A,\chi,\tau)$, where $\chi:A\times A\to\mathbb K^\times$ and $\tau\in\mathbb K^\times$ are used to determine the associator isormorphisms and must satisfy certain conditions.  This landmark paper stands apart in the field of fusion categories as one of the few times such a categorification has ever been done explicitly by hand. Tambara and Yamagami work over arbitrary fields $\mathbb K$ but assume that all of their simple objects $X$ are split, i.e. $\End(X)\cong\mathbb K$. Their main result is the following.

\begin{theorem}\cite[Theorem 3.2]{TAMBARA1998692}\label{OGTYTheorem}
    Given a triple $(A,\chi,\tau)$, where $A$ is a finite group, $\chi:A\times A\to\mathbb K^\times$ is a nondegenerate symmetric bicharacter, and $\tau\in\mathbb K$ satisfies $\tau^2=\sfrac{1}{|A|}$, there exists a split Tambara-Yamagami category $\mathcal C(A,\chi,\tau)$ with fusion ring $\TY(A)$ and associators given below.
    
    \begin{align*}
    	    \alpha_{a,b,c}&=\id_{abc},\\
    	    \alpha_{a,b,m}=\alpha_{m,b,c}&=\id_{m},\\
    	    \alpha_{a,m,c}&=\chi(a,c)\cdot\id_{m},\\
    	    \alpha_{a,m,m}=\alpha_{m,m,c}&=\id_{m\otimes m},\\
    	    \alpha_{m,b,m}&=\bigoplus_{a\in A}\chi(a,b)\cdot\id_{a},\\
    		  \alpha_{m,m,m}&=\left(\frac{\tau}{\chi(a,b)}\cdot\id_m\right)_{a,b}:\bigoplus_{a\in A}m\longrightarrow\bigoplus_{b\in A}m\,.
    \end{align*}
    
    Furthermore, any Tambara-Yamagami category is monoidally equivalent to \newline $\mathcal C(A,\chi,\tau)$ for such a triple.  Two categories $\mathcal C(A,\chi,\tau)$ and $\mathcal C(A',\chi',\tau')$ are monoidally equivalent if and only if $\tau=\tau'$ and there exists an isomorphism $f:A\to A'$ such that for any $a,b\in A$, $\chi'\big(f(a),f(b)\big)=\chi(a,b)$.
\end{theorem}

\begin{remark}
 Notice that the nondegeneracy of the bicharacter implies that $A$ must be abelian. So for the fusion ring $\TY(A)$ to be categorifiable, $A$ must be an abelian group.
\end{remark}

The assumption that all simple objects must be split is automatic if $\mathbb K$ is an algebraically closed field. In order to understand the full picture in the non-algebraically closed setting, we generalize Tambara-Yamagami fusion categories to include the possibility of non-split simple objects. A priori it is not obvious that such categories should exist, so let us consider two fusion categories over $\mathbb R$ whose fusion rules are similar to Tambara-Yamagami fusion rules.

\begin{example}\label{TYEg1}
    Consider the category $\Rep_{\bb R}(\bb Z/4\bb Z)$ of finite dimensional real representations of $\bb Z/4\bb Z$.  This category has two invertible objects $\1$ and $X$, and there is an additional irreducible representation $V$.  If we let $t$ be a generator for $\mathbb Z/4\mathbb Z$, then each of these representations can be described by writing down the coordinate matrix by which $t$ acts in a chosen basis.  In this way, the irreducible representations $\1, X$, and $V$ are described as follows:
    \begin{align*}
        \1&=\big(\mathbb R^1\;,\;t\mapsto[1]\big)\,,\\
        X&=\big(\mathbb R^1\;,\;t\mapsto[-1]\big)\,,\;\text{ and}\\
        V&=\left(\mathbb R^2\;,\;t\mapsto\big[\,{\begin{smallmatrix}0&-1\\1&0\end{smallmatrix}}\,\big]\right)\,.
    \end{align*}
    From these definitions, it can be shown that 
    \[X\otimes V\cong V\otimes X\cong V\;\text{ and }\;V\otimes V\cong2\cdot(\1\oplus X)\,.\]
    These fusion rules show that $\Rep_{\bb R}(\bb Z/4\bb Z)$ looks similar to a Tambara-Yamagami category but with some key differences. The first difference is that $\End(V)\cong\bb C$, and the second is that $V\otimes V$ has two copies of each invertible.
\end{example}

\begin{example}\label{TYEg2}
    Let $\mathbb H=\mathbb R\langle i, j, k\rangle/(i^2=j^2=k^2=ijk=-1)$ be the quaternion algebra, and let $Q_8=\{\pm1, \pm i, \pm j, \pm k\}\subset\mathbb H^\times$ be the quaternion group of order 8.  Consider the category $\Rep_{\bb R}(Q_8)$ of finite dimensional real representations of $Q_8$. This category has four invertible objects $\1$, $I$, $J$ and $K$.  The representations $I$, $J$, and $K$ are one dimensional, and are determined by the requirement that $i$, $j$, and $k$ respectively act trivially.  There is an additional irreducible representation $H=\mathbb H$ given by the quaternion algebra itself, where $Q_8$ acts by left multiplication.  From these definitions, it can be shown that the fusion rules are
    $X\otimes H\cong H\otimes X\cong H,$
    for any invertible object $X$, and 
    \[H\otimes H\cong4\cdot\big(\1\oplus I\oplus J\oplus K\big)\,.\]
    In this example, we encounter two aspects that make it slightly different from a Tambara-Yamagami category.  Firstly,  $\End(H)\cong\bb H$. Secondly, $H\otimes H$ has four copies of every invertible object.
\end{example}


The above examples have a striking similarity to Tambara-Yamagami categories, and they appear to differ from a Tambara-Yamagami fusion ring $\TY(A)$ in a predictable way. In the next sections, we pursue the study of fusion rings and fusion categories similar to the ones described in the examples above.

The results of our article achieve this classification over $\mathbb R$ and demonstrate that these non-split versions of Tambara-Yamagami categories are indeed very common.

\section{Fusion Categories over the Reals}\label{Sec:FusionCats/R}

The ideas of this section were originally developed in the second author's thesis \cite{sanfordThesis} in the more general setting of non\textendash algebraically closed fields $\mathbb K$.  Here we specialize to the case $\mathbb K=\mathbb R$ and present facts such as Proposition \ref{Thm:GaloisGrading} which are unique to the real numbers.

\subsection{Schur's Lemma}
For  $\mathbb K$-linear abelian categories, we say that a nonzero object is simple if it has no nontrivial subobjects.  The following lemma is a modern adaptation of a representation-theoretic result of Schur.
\begin{lemma}[Schur's Lemma]\label{SchursLemma}
Let $X$ and $Y$ be simple objects in a $\mathbb K$-linear abelian category.  If $X\ncong Y$, then $\Hom(X,Y)=0$ and $\End(X)$ is a division algebra.
\end{lemma}
Over algebraically closed fields, the only finite dimensional division algebra is the field itself, and fusion categories over algebraically closed fields have been extensively studied, see for example \cite{etingofFusionCategories2005}.  The possibility of having $\End(X)$ be a nontrivial division algebra is the primary source of new phenomena for fusion categories over non\textendash algebraically closed fields.
\begin{definition}\label{Def:Non-Split}
    A simple object $X$ in a $\mathbb K$-linear abelian category is said to be \emph{split}, or split simple, if $\End(X)\cong\mathbb K$. 
    Otherwise, the simple object is called \emph{non-split}.  A category is said to be \emph{split} if all of its simple objects are split.
\end{definition}

Finite groups give rise to families of examples of fusion categories. We will denote by $\Rep_{\mathbb K} (G)$ the category of finite-dimensional representations over $\mathbb K$ of $G$.
The simple objects in this category are the irreducible representations of $G$. Example \ref{TYEg1} describes $\Rep_{\mathbb R}(\mathbb Z/4\mathbb Z)$, which has a non-split irreducible representation $V$ that 
has $\End(V)\cong\mathbb C$. In Example \ref{TYEg2} the category $\Rep_{\mathbb R}(Q_8)$ is considered, and it has a non-split simple object $H$ that 
has $\End(H)\cong\mathbb H$.

These examples are generic in a certain sense.  When working over the real numbers, there are only two ways for a simple object $X$ to be non-split, that is, $\End(X)\cong\mathbb C$ or $\End(X)\cong\mathbb H$.
This is a consequence of the following well-known result of Frobenius.

\begin{theorem}[\cite{Frobenius1877}]\label{Thm:FrobeniusDivAlg/R}
    Any finite dimensional division algebra over the real numbers must be isomorphic to $\mathbb R$, $\mathbb C$, or $\mathbb H$.
\end{theorem}


\begin{definition}
    A simple object is said to be real, complex, or quaternionic if \linebreak$\End(X)\cong\mathbb R$, $\mathbb C$, or $\mathbb H$ respectively.
\end{definition}

Despite the fact that all possible division algebras can occur, there are restrictions on which simples can have which endomorphism algebras when the category is monoidal.
Invertible objects, objects $a$ for which $a^*\otimes a\cong\1$, are in particular constrained as the following proposition shows.

\begin{proposition}\label{Prop:AllSameDivAlg}
    All invertible objects have isomorphic endomorphism algebras in a fusion category over an arbitrary field $\mathbb K$.  Moreover, this common algebra is a finite-dimensional field extension of $\mathbb K$.
\end{proposition}

\begin{proof}
    If $g$ is an invertible object, then $(-)\otimes g$ is an equivalence, and so
    \[\End(\1)\cong\End(\1\otimes g)\cong\End(g)\,.\]
    Thus all endomorphism algebras of invertible objects are isomorphic to one another.
    
    The Eckmann-Hilton argument forces $\End(\1)$ to be commutative. Since $\End(\1)$ is a finite-dimensional commutative division algebra that contains $\mathbb K$, the claim follows.
\end{proof}

Proposition \ref{Prop:AllSameDivAlg} has an interesting dichotomy as a corollary.  To elaborate, we will need some terminology.  For any object $X$, we can use the isomorphism $\ell_X:\1\otimes X\to X$ to turn endomorphisms of $\1$ into endomorphisms of $X$.  When the base field admits nontrivial division algebras this embedding of $\End(\1)$ into $\End(X)$ may not be obvious.

\begin{definition}\label{def: lambda and rho}
    Let $e\in\End(\1)$.  The endomorphisms $\lambda_X(e),\rho_X(e):X\to X$ are defined as the compositions below
    \[\begin{tikzcd}[ampersand replacement=\&]
    	\& X \& X \\
    	{\1\otimes X} \&\&\& X\otimes\1 \\
    	{\1\otimes X} \&\&\& X\otimes\1 \& {.} \\
    	\& X \& X
    	\arrow["{\lambda_X(e)}"', from=1-2, to=4-2]
    	\arrow["{\rho_X(e)}", from=1-3, to=4-3]
    	\arrow["{\ell_X^{-1}}"', from=1-2, to=2-1]
    	\arrow["{e\otimes\id_X}"', from=2-1, to=3-1]
    	\arrow["{\ell_X}"', from=3-1, to=4-2]
    	\arrow["{r_X^{-1}}", from=1-3, to=2-4]
    	\arrow["{\id_X\otimes e}", from=2-4, to=3-4]
    	\arrow["{r_X}", from=3-4, to=4-3]
    \end{tikzcd}\]
    These define algebra embeddings
    \[\lambda_X,\rho_X:\End(\1)\hookrightarrow\End(X),\]
    that are called the left and right embeddings for $X$.  The naturality of the unitors $\ell$ and $r$ imply that the embeddings $\lambda_X$ and $\rho_X$ factor through the inclusion of the center, as in the diagram below
    \[\begin{tikzcd}[ampersand replacement=\&]
    	\&\& {\End(X)} \\
    	{\End(\1)} \&\& {Z(\End(X))\;,}
    	\arrow[hook, from=2-3, to=1-3]
    	\arrow["{\lambda_X^0}"', hook, from=2-1, to=2-3]
    	\arrow["{\lambda_X}", hook, from=2-1, to=1-3]
    \end{tikzcd}\]
    and similarly for $\rho_X$.
\end{definition}

\begin{corollary}\label{Cor:AllRealOrAllComplex}
    In a fusion category over $\mathbb R$, either all invertible objects are real, or all invertible objects are complex.  Moreover, if the invertible objects are complex then all simple objects are complex.
\end{corollary}

\begin{proof}
    By combining the Frobenius Theorem (Theorem \ref{Thm:FrobeniusDivAlg/R}) with Proposition \ref{Prop:AllSameDivAlg}, it follows that $\1$ is either real or complex, and all invertibles must be of matching type.
    If a simple object $X$ is real or quaternionic then $Z\big(\End(X)\big) = \mathbb R$.  Since $\lambda_X^0:\End(\1)\hookrightarrow\mathbb R$ is an algebra embedding, we see that real and quaternionic objects can only exist when $\1$ is real.  Thus if $\1$ is complex, every simple object must also be complex.
\end{proof}

Corollary \ref{Cor:AllRealOrAllComplex} raises the following question. If a fusion category over $\mathbb R$ has the property that all of its simple objects are complex, would that mean that the category is also fusion over $\mathbb C$?  The following example shows that the answer is no.

\begin{example}\label{Eg:CCBim}
    Consider the complex numbers $\mathbb C$ as an algebra over the real numbers.  Let $\mathcal C=(\mathbb C,\mathbb C)$-bim be the category of finite dimensional bimodules for this algebra.  This category is equivalent to the category of modules for the algebra $\mathbb C\otimes_{\mathbb R}\mathbb C$.
    This category is monoidal, with the tensor product being the relative tensor product $\otimes_{\mathbb C}$, and monoidal unit $\1_{\mathcal C}=\mathbb C$.  This category has another simple bimodule $\overline{\mathbb C}$, where the left and right actions of $\mathbb C$ differ by complex conjugation.  It can easily be shown that $\End(\overline{\mathbb C})\cong\mathbb C$.

    Thus all simple objects of $\mathcal C$ are complex.  
    However, this category is not fusion over $\mathbb C$ because the tensor product fails to be $\mathbb C$-bilinear.  To see this consider a complex number $c$ and observe that $\id_{\overline{\mathbb C}}\otimes c\cdot\id_{\overline{\mathbb C}}\;=\;\overline{c}\cdot\id_{\overline{\mathbb C}}\otimes\id_{\overline{\mathbb C}}$ as morphisms in $\End(\overline{\mathbb C}\otimes\overline{\mathbb C})$.
\end{example}

\subsection{Galois Nontrivial Objects}\label{sSec:GaloisNontrivial}

The content of this subsection will not be used until Section \ref{Sec:ComplexGalois}.  Any reader primarily interested in the case where $\End(\1)\cong\mathbb R$ may safely skip ahead to Section \ref{Sub:Non-SplitTYCats}.

The object $\overline{\mathbb C}$ in Example \ref{Eg:CCBim} is what is known as a Galois nontrivial object.  That is, the conjugating complex bimodule $\overline{\mathbb C}$ is an object for which $\lambda_{\overline{\mathbb C}}\neq\rho_{\overline{\mathbb C}}$.  This is a phenomenon that cannot occur when working over an algebraically closed field, so we give it a name.

\begin{definition}\label{Def:Galois(Non)Triv}
    An object $X$ in a fusion category  is called \emph{Galois trivial} if $\lambda_X=\rho_X$. Otherwise, we say that $X$ is \emph{Galois nontrivial}.
\end{definition}

In general, it is possible for $\im(\lambda_X)$ and $\im(\rho_X)$ to be distinct subalgebras of $\End(X)$, but over $\mathbb R$ such issues do not occur.

\begin{proposition}\label{Prop:GaloisNontrivialsOverR}
    Let $\mathcal C$ be a fusion category over $\mathbb R$. If $\mathcal C$ has Galois nontrivial simple objects then all simple objects are necessarily complex and the left and right embeddings of Galois nontrivial simple objects differ by complex conjugation.
\end{proposition}

\begin{proof}
    Observe that the linearity assumptions on fusion categories imply that all objects are automatically Galois trivial whenever the unit $\1$ is split.  When working over $\mathbb R$ this means that $\1$ must be complex for Galois nontrivial objects to exist.  By Corollary \ref{Cor:AllRealOrAllComplex}, all simple objects must be complex.

    For a given simple object $X$, $\lambda_X$ and $\rho_X$ are both algebra automorphisms of $\mathbb C$ that happen to fix $\mathbb R$.  Knowing that the embeddings are isomorphisms makes Galois nontriviality of $X$ equivalent to the statement that $\lambda_X^{-1}\circ\rho_X\neq\id_{\mathbb C}$, therefore this automorphism must be complex conjugation.
\end{proof}

Given a fusion category $\mathcal C$, we define $\mathcal C_0$ as the full subcategory generated under direct sums by the Galois trivial simple objects, and $\mathcal C_1$ as the full subcategory generated by the Galois nontrivial objects.  In this way, we obtain a $\mathbb Z/2\mathbb Z$-grading $\mathcal C\simeq\mathcal C_0\oplus\mathcal C_1$ as $\mathbb R$-linear abelian categories.  We will demonstrate that this grading respects the monoidal structure as well.

\begin{lemma}\label{Lem:TrivalentGaloisCompatibility}
    Let $\mathcal C$ be a monoidal $\mathbb K$-linear abelian category.  Let $X, Y,$ and $Z$ be simple objects in $\mathcal C$ and let $f:Z\to X\otimes Y$ be a morphism.  If $\lambda_X$, $\lambda_Y$, and $\lambda_Z$ are invertible, and $f\neq0$, then
    \[\lambda_X^{-1}\rho_X\lambda_Y^{-1}\rho_Y\;=\;\lambda_Z^{-1}\rho_Z.\]
\end{lemma}

\begin{proof}
This follows from naturality and the triangle axiom.
\end{proof}

\begin{theorem}\label{Thm:GaloisGrading}
    All fusion categories over $\mathbb R$ that contain Galois nontrivial objects necessarily admit a faithful grading by the group $\text{Gal}(\mathbb C/\mathbb R)\cong\mathbb Z/2\mathbb Z$.
\end{theorem}

\begin{proof}
    Let $X$ be an object in $\mathcal C_i$ and $Y$ be an object in $\mathcal C_j$.  For any simple summand $Z$ of $X\otimes Y$, we can find some simple summands $X_0\hookrightarrow X$ and $Y_0\hookrightarrow Y$ such that $Z$ is a simple summand of $X_0\otimes Y_0$.  Since each of the $\mathcal C_k$ are full subcategories, $X_0$ is in $\mathcal C_i$ and $Y_0$ is in $\mathcal C_j$.  By Proposition \ref{Prop:GaloisNontrivialsOverR}, $\lambda_{X_0}^{-1}$, $\lambda_{Y_0}^{-1}$, and $\lambda_{Z}^{-1}$ all exist, so we may apply Lemma \ref{Lem:TrivalentGaloisCompatibility} to the inclusion morphism $f=\iota:Z\hookrightarrow X_0\otimes Y_0$ to see that $Z$ is in $\mathcal C_{i+j}$.  Since $Z$ was arbitrary, all simple summands of $X\otimes Y$ are contained in $\mathcal C_{i+j}$, so the entire object $X\otimes Y$ must be in $\mathcal C_{i+j}$ as well.

    Finally, the definition of $\mathcal C_1$ immediately implies that the existence of Galois nontrivial objects is equivalent to the faithfulness of the grading.
\end{proof}

\begin{definition}
      We will refer to the grading established in Theorem \ref{Thm:GaloisGrading} as the \emph{Galois grading}.
\end{definition}



To end this section, we record a corollary of Theorem \ref{Thm:GaloisGrading} that will be helpful in Section \ref{Sec:ComplexGalois}.

\begin{corollary}\label{Cor:XXIsC0}
    If $\mathcal C$ is a fusion category over $\mathbb R$ that contains Galois nontrivial objects, then for any simple object $X$ in $\mathcal C$, the object $X\otimes X$ is in $\mathcal C_0$.
\end{corollary}

\section{Tambara-Yamagami Fusion Categories: Non-Split Case}\label{Sub:Non-SplitTYCats}

We would like to investigate non-split generalizations of the (split) Tambara-\newline Yamagami categories described in Subsection \ref{Sub:SplitTYCats}. In particular, we will focus on the case where $\mathbb K = \mathbb R$. In making our generalization, the features that we would like to preserve are the following:
\begin{enumerate}
    \item the set of (isomorphism classes of) simple objects consist of a group $A$ of invertible objects, together with a single self-dual simple object $m$, and
    \item the object $m\otimes m$ is a direct sum of invertible objects, that is, the multiplicity of $m$ in $m\otimes m$ is 0.
\end{enumerate}

A natural starting point would be to allow various simple objects to be non-split, but not all division algebras are possible. We use the results from Section \ref{Sec:FusionCats/R} to narrow these options down to only three possibilities. In particular, Theorem \ref{Thm:FrobeniusDivAlg/R} implies that when working over $\mathbb R$ there are only three ways for a simple object to be non-split. Moreover, Corollary \ref{Cor:AllRealOrAllComplex} shows that either all the \emph{invertible} objects are real or all the \emph{simple} objects are complex.

If all the invertible objects are real, then $m$ is the only simple object that could be non-split, and so $m$ can either be complex or quaternionic.
If all simple objects are complex, either the category is fusion over $\mathbb C$ or not. The original theorem of Tambara and Yamagami (see Theorem \ref{OGTYTheorem}) already covers the case where $\mathcal C$ is fusion over $\mathbb C$, and so we assume that $\mathcal C$ is only fusion over $\mathbb R$. The results of Subsection \ref{sSec:GaloisNontrivial} 
show that $\mathcal C$ must contain Galois nontrivial simple objects and hence, by Theorem \ref{Thm:GaloisGrading}, $\mathcal C$ is faithfully Galois graded. Moreover, Corollary \ref{Cor:XXIsC0} implies that the object $m\otimes m$ is Galois trivial, and by our assumption on the fusion rules, this forces all the invertible objects to be Galois trivial.  Then, since there must be at least one Galois nontrivial object, it must be $m$.

Summarizing, there are three new possibilities for non-split Tambara-Yamagami categories over $\mathbb R$:

\begin{enumerate}[itemsep=2mm,leftmargin=20mm, rightmargin=38mm]
    \item[Case 1:] all invertible objects are real and $m$ is quaternionic,
    \item[Case 2:] all invertible objects are real and $m$ is complex, or
    \item[Case 3:] all simple objects are complex and $m$ is the unique Galois nontrivial simple.
\end{enumerate}

Before starting with the analysis of each of the different cases, we highlight
some important aspects that are common to all three contexts.

In each case we begin by fixing an isomorphism between $\End(m)$ and the relevant division algebra $\mathbb D\in\{\mathbb C,\mathbb H\}$. We use this fixed isomorphism to identify elements of the algebra $e\in\mathbb D$ with endomorphisms $e:m\to m$ in the category.  In the complex Galois case, we further identify $\End(\1)$ with $\End(m)=\mathbb C$ using the left embedding $\lambda_m$ (see Definition \ref{def: lambda and rho}), and thus force $\rho_m$ to be complex conjugation.

We consider the following $\mathbb R$-vector spaces
\[
    \End(m),\quad \Hom(a\otimes m,m),\quad \Hom(m\otimes a,m),\quad \Hom(m\otimes m,a),
\]
which are all isomorphic by rigidity of the fusion category.
This allows us to compute the multiplicity of each invertible object in $m\otimes m$
\[\End(m) \cong\Hom(a,m\otimes m) \cong \Hom\Big(a\,,\,\bigoplus_{b\in A} b^{\oplus n_b}\Big) \cong \bigoplus_{b\in A} \delta_{a,b}\End(b)^{\oplus n_b}\cong \End(a)^{\oplus n_a}.\]
By Proposition \ref{Prop:AllSameDivAlg}, $\End(a)\cong\End(\1)$ is a field, so for every $a\in A$, the multiplicity of $a$ in $m\otimes m$ is the number $n_a = \dim_{\End(\1)}(\End(m))$.
This gives 
the following variation of the split fusion rules considered in \cite{TAMBARA1998692} for the non-split cases:
\begin{equation*}
	m\otimes m = 
 \dim_{\End(\1)}\big(\End(m)\big)\cdot\bigoplus_{a \in A} a\;.
\end{equation*}

In each of the following sections we proceed first by choosing basis vectors for the hom spaces and then by writing down the coordinate matrix of the associators in terms of the chosen basis. Explicitly, precomposition with the associator $\alpha_{W,X,Y}:(W\otimes X)\otimes Y\to W\otimes (X\otimes Y)$ produces a map on hom spaces
\[\Hom(\alpha_{W,X,Y}\,,\, Z):\Hom\big(W\otimes(X\otimes Y)\,,\,Z\big)\to\Hom\big((W\otimes X)\otimes Y\,,\,Z\big)\,,\]
for each target object $Z\in \mathcal C$. Here we introduce the notation we will use, that follows Tambara and Yamagami's original notation from \cite{TAMBARA1998692}.
There are isomorphisms
\[\Hom\big(W\otimes(X\otimes Y)\,,\,Z\big)\cong \bigoplus_{U\text{ simple}}\Hom(W\otimes U\,,\,Z)\mathop{\otimes}\limits_{\End(U)}\Hom(X\otimes Y\,,\,U)\,,\]
\[\Hom\big((W\otimes X)\otimes Y\,,\,Z\big)\cong\bigoplus_{V\text{ simple}}\Hom(V\otimes Y\,,\,Z)\mathop{\otimes}\limits_{\End(V)}\Hom(W\otimes X\,,\,V)\,.\]
By composing these with the map $\Hom(\alpha_{W,X,Y},Z)$, we arrive at a more concrete description of the associator.
\begin{definition}\label{Def:TetrahedralTrans}
    The tetrahedral transformation $\tetr{W}{X}{Y}{Z}$ is the composition indicated in the following commutative diagram
    \[\begin{tikzcd}[ampersand replacement=\&, column sep=20]
    	{\Hom\big(W\otimes(X\otimes Y)\,,\,Z\big)} \&\& {\bigoplus_{U}\Hom(W\otimes U\,,\,Z)\mathop{\otimes}\limits_{\End(U)}\Hom(X\otimes Y\,,\,U)} \\
    	{\Hom\big((W\otimes X)\otimes Y\,,\,Z\big)} \&\& {\bigoplus_{V}\Hom(V\otimes Y\,,\,Z)\mathop{\otimes}\limits_{\End(V)}\Hom(W\otimes X\,,\,U)\,.}
    	\arrow["\cong"', from=1-3, to=1-1]
    	\arrow["{\Hom(\alpha_{W,X,Y},Z)}", from=1-1, to=2-1]
    	\arrow["\cong", from=2-1, to=2-3]
    	\arrow["{\tetr{W}{X}{Y}{Z}}", from=1-3, to=2-3]
    \end{tikzcd}\]
\end{definition}

\begin{remark}
    The name tetrahedral transformations comes from \cite{TAMBARA1998692}, and is not standard.  In more modern language these are often called $F$-symbols or $F$-matrices (see e.g. \cite{bondersonOnInvariants}).
\end{remark}


The notation that is common throughout the three next sections
follows the conventions in \cite{TAMBARA1998692}.  In each section, corresponding to each of the three cases above, there will be a preferred way of constructing nonzero (and hence surjective) morphisms 
\begin{align*}
    [a,b]&\in\Hom(a\otimes b, ab)\,,\\
    [a,m]&\in\Hom(a\otimes m,m)\,,\\
    [m,a]&\in\Hom(m\otimes a,m)\,, \text{ and}\\
    [a]&\in\Hom(m\otimes m,a)\,.
\end{align*}

By Schur's Lemma (Lemma \ref{SchursLemma}), the first three will be isomorphisms, and the map $[a]$ will only be surjective.  Once $[a]$ is chosen, by semisimplicity, there is a splitting $[a]':a\to m\otimes m$.

In the first case, when $\1$ is real and $m$ is quaternionic, we set $S=\{1, i, j, k\}$. In the second case, when $\1$ is real and $m$ is complex, we set $S=\{1,i\}$.  In the third case, we can set $S=\{1\}$.  In all cases, we define the following useful map $[a]^\dagger$.
 
\begin{definition}\label{Def:TheDagger} Let $a\in A$.  If $m$ is Galois nontrivial, then set $[a]^\dagger=[a]'$.  In the other cases, proceed with the following construction.

The $\mathbb R$-linear map $T:\End(m)\to\mathbb R$ is given by the formula
\[[a](\id_m\otimes e)[a]'\;=\;T(e)\cdot\id_{a}\,,\]
for $e\in \End(m)$.

The map $[a]^\dagger:\1\to m\otimes m$ is given by the formula
    \[[a]^\dagger:=\frac{\sum_{s\in S}T(s)\cdot(\id_m\otimes s)[a]'}{\sum_{r\in S}T(r)^2}\,.\]
\end{definition}

Notice that the denominator in the definition of $[a]^\dagger$ is nonzero because $T(1)=1$. A direct consequence of this definition is the following property.

\begin{proposition}\label{Prop:RealPartWDagger}
    Suppose $\1$ is real and $m$ is either complex or quaternionic (case 1 or case 2 above).  The map $[a]^\dagger$ satisfies the formula $$[a](\id_{m}\otimes e)[a]^\dagger=\Re(e)\cdot\id_a\,,$$ where $\Re(e)$ is the real part of $e\in\End(m)$.
\end{proposition}

Using the  maps $[a]^\dagger$, we produce orthogonal projections $$(\id_m\otimes s)[a]^\dagger[a](\id_m\otimes \overline{s})\,:\,m\otimes m\to m\otimes m\,,$$ and, in this way, we identify summands of $m\otimes m$ with pairs $(a,s)$, where $a\in A$ and $s\in S$. Notice that $[a]^\dagger$ is uniquely determined by $[a]$, so changes to $[a]$ will alter $[a]^\dagger$ accordingly.

Using the preferred vectors $[a,m]$, $[m,a]$ and $[a]$, we construct bases for the hom spaces as follows:
\begin{align*}
    \{[a,m](\id_a\otimes s)\}_{s\in S}&\text{ for the space }\Hom(a\otimes m,m)\,,\\
    \{[m,a](s\otimes \id_a)\}_{s\in S}&\text{ for the space }\Hom(m\otimes a,m)\,,\text{ and}\\
    \{[a](\id_m\otimes s)\}_{s\in S}&\text{ for the space }\Hom(m\otimes m,a)\,.
\end{align*}

In these bases, the tetrahedral transformations are determined by what they do on simple tensors of the vectors $[a,m]$, $[m,a]$, and $[a]$.  The general naming scheme of the matrix coefficients of the tetrahedral transformations is shown in the table below.

\begin{center}
\begin{tabular}{|c|c|c|c|c|}\hline
     Tetrahedral & $\tetr{a}{b}{c}{abc}$ & $\tetr{m}{a}{b}{m}$ & $\tetr{a}{m}{b}{m}$ & $\tetr{a}{b}{m}{m}$ \\\hline
     Coefficient & $\alpha(a,b,c)$ & $\alpha_1(a,b)$ & $\alpha_2(a,b)$ & $\alpha_3(a,b)$ \\\hline\hline
     Tetrahedral & $\tetr{a}{m}{m}{b}$ & $\tetr{m}{a}{m}{b}$ & $\tetr{m}{m}{a}{b}$ & $\tetr{m}{m}{m}{m}$\\\hline
     Coefficient & $\beta_1(a,b)$ & $\beta_2(a,b)$ & $\beta_3(a,b)$ & $\big(\gamma(a,b)\big)_{a,b}$\\\hline
\end{tabular}
\end{center}



           
            
\begin{remark}\label{Rem:GammaMatrixDim}
    The function $\alpha=\alpha(a,b,c)$ will take values in $\End(\1)$, and all of the $\alpha_i$'s and $\beta_j$'s will take values in $\End(m)$.  The hom space
    \[\Hom\big(m\otimes(m\otimes m),m\big)\cong\Hom\big((m\otimes m)\otimes m,m\big)\cong\big(\End(m)\otimes_{\End(\1)}\End(m)\big)^{\oplus|A|}\]
    has dimension $n=|A|\cdot\dim_{\End(\1)}(\End(m))^2$ over the field $\End(\1)$.  Thus in general $\gamma$ is described by a matrix in $GL_n(\End(\1))$.  These associator coefficients appear to hold a large amount of information but, in each case, naturality allows for a significant reduction in complexity.
\end{remark}

The combinatorics imply that there are 16 different types of pentagon equations to be solved in each case.  In order to determine when two of our categories are monoidally equivalent, there are 4 matrix coefficients for the tensorators, and they are subject to 8 coherence equations.  The tensorator naming conventions are shown in the table below.

\begin{center}
\begin{tabular}{|c|c|c|c|c|}\hline
     Tensorator & $J_{a,b}$ & $J_{a,m}$ & $J_{m,b}$ & $J_{m,m}$ \\\hline
     Coefficient & $\theta(a,b)^{-1}$ & $\varphi(a)^{-1}$ & $\psi(b)^{-1}$ & $\big(\,\omega(a)^{-1}\,\big)_a$ \\\hline
\end{tabular}
\end{center}

\begin{remark}\label{Rem:InversesInTheTensorators}
    Note the unfortunate presence of inverses.  These inverses appear because we aim to align our notation with the change of basis transformations in \cite[page 700]{TAMBARA1998692}. These change of basis transformations are monoidal equivalences, in which the coefficients more naturally appear on the opposite side of the equation.
\end{remark}

\begin{note}[Rightmost Factor Convention]\label{Note:RightmostFactor}
    When writing down tetrahedral transformations in terms of the associator coefficients, the formulas become very wide.  For the sake of compactness and legibility, we develop some conventions.  Whenever an element of $\End(m)$ appears immediately to the right of a morphism, this denotes precomposition with that morphism on the rightmost factor of $m$ that appears in the input, tensored with the appropriate number of identity morphisms on either side.  For example, $[a,m]t = [a,m] \circ (\id_a\otimes t)$ and $[m,a]t = [m,a] \circ (t\otimes \id_a)$.  For $[a]$, the case where there are two factors of $m$ in the input, $[a]t$ will denote $[a] \circ (\id_m\otimes t)$, and we reserve the notation $[a]\triangleleft t$ for the composition $[a] \circ (t\otimes\id_m)$.  The category is assumed to be linear over $\bb R$, so we will simply write $r\cdot-$ to indicate scalar multiplication by a real number $r\in\bb R$.
    
    In the process of deriving the pentagon equations, it becomes necessary to bring all of the coefficients into either $\End(\1)$ or $\End(m)$.  Once everything lies within a single vector space, we can compare coefficients of our basis vectors to arrive at the desired equations.  Since some of the pentagon equations involve tensor products of multiple copies of $m$, it is necessary to make an arbitrary choice of where to put all the coefficients.  In keeping with the above conventions, all morphisms in $\End(m)$ will be moved to the rightmost factor of $m$ that appears in the \emph{input} of the tetrahedral transformation (see Definition \ref{Def:TetrahedralTrans}).  This can always be achieved by passing morphisms across a relative tensor product, or applying Relations \ref{nice-basis}, Definition \ref{Def:The-g-Symbol}, Definition \ref{Def:The-a-symbol}, or possibly iterated compositions thereof.
\end{note}

    \section{Analysis of the Real-Quaternionic Case\label{Sec:Real/Quaternionic}}

We will now construct the non-split Tambara-Yamagami categories $\mathcal C_{\mathbb H}(A,\chi,\tau)$.

\subsection{Choosing a preferred basis}

The following observation makes the choice of a basis simpler.

\begin{proposition}\label{Prop:basis}
Let $V$ be an $(\mathbb H,\mathbb
H)$-bimodule. If $V$ is 4-dimensional as an $\bb{R}$-vector space then there exists a nonzero $v \in V$ such that $h . v = v . h$,
for all $h \in \mathbb{H}$.
\end{proposition}

\begin{proof}
    An $(\mathbb H,\mathbb H)$-bimodule is the same as an $\mathbb H\otimes_{\mathbb R}\mathbb H^{op}$-module.  Since $\mathbb H\otimes_{\mathbb R}\mathbb H^{op}\cong M_4(\mathbb R)$ as algebras, there is a unique simple $(\mathbb H,\mathbb H)$-bimodule up to isomorphism. Any such bimodule is simple if and only if it is 4-dimensional over $\mathbb R$.  Let us choose some bimodule isomorphism $\phi:\mathbb H\to V$ from the trivial bimodule to our given bimodule $V$.  Then, the vector $v:=\phi(1)$ has the desired property. 
\end{proof}

Consider the following 4-dimensional (as $\mathbb R$-vector spaces) hom spaces
\[
    \Hom(a\otimes m,m),\quad \Hom(m\otimes a,m),\quad \text{and }\;\Hom(m\otimes m,a).
\]

Proposition \ref{Prop:basis} shows that there is always a choice 
of non-zero morphisms $[a,m]$, $[m,a]$, and $[a]$ such that the quaternions commute with them.

One subtlety here is that the space $\Hom(m\otimes m,a)$ is most naturally a right $\mathbb H\otimes_{\mathbb R}\mathbb H$-module.  This can be thought of as an $(\mathbb H^{op},\mathbb H)$-bimodule.  Quaternionic conjugation $h\mapsto\overline{h}$ provides an isomorphism $\mathbb H^{op}\to\mathbb H$, and this can be used to transform $\Hom(m\otimes m,a)$ into an $(\mathbb H,\mathbb H)$-bimodule.
By Proposition \ref{Prop:basis}, there is a preferred vector $[a]$.  Since we needed to apply quaternionic conjugation to one of the actions, the resulting `commutation' property for $[a]$ involves conjugation.


Summarizing, bases have been chosen for the 
hom spaces using the morphisms $[a,m], [m,a],$ and $[a]$ such that
\begin{align}\label{nice-basis}
\begin{split}
	[a,m](id_a \otimes h) =&\ h[a,m],\\
	[m,a](h \otimes id_a) =&\ h[m,a],\\
	[a](id_m \otimes h) =&\ [a](\bar{h} \otimes id_m).
 \end{split}
\end{align}

\subsection{The associators}\label{sub:QuatAssoc}

With our conventions established, the tetrahedral transformations are as follows
\begin{align}
	\tetr{a}{b}{c}{abc}:& [a,bc]\otimes[b,c] \mapsto \alpha(a,b,c)\cdot[ab,c]\otimes[a,b] ,\nonumber\\
	\tetr{a}{b}{m}{m}:&   [a,m]\otimes [b,m] \mapsto \Big([ab,m]\alpha_3(a,b)\Big)\otimes[a,b],\nonumber\\
	\tetr{a}{m}{b}{m}:&   [a,m]\otimes[m,b]  \mapsto [m,b]\otimes\Big([a,m]\alpha_2(a,b)\Big),\nonumber\\
	\tetr{m}{a}{b}{m}:&   [m,ab]\otimes[a,b] \mapsto  [m,b]\otimes\Big([m,a]\alpha_1(a,b)\Big),\nonumber\\
	\tetr{a}{m}{m}{b}:&   [a,a^{-1}b]\otimes[a^{-1}b]  \mapsto \Big([b]\beta_1(a,b)\Big)\otimes[a,m] ,\nonumber\\
	\tetr{m}{a}{m}{b}:&   [b]\otimes[a,m]  \mapsto \Big([b]\beta_2(a,b)\Big)\otimes[m,a] ,\nonumber\\
	\tetr{m}{m}{a}{b}:&   [b]\otimes[m,a]  \mapsto [ba^{-1},a]\otimes \Big([ba^{-1}]\beta_3(a,b)\Big),\nonumber\\
	\tetr{m}{m}{m}{m}:&   [m,a]\otimes[a]\mapsto \sum_{b\in A\,,\,s,t\in S} \gamma(a,b,s,t)\Big([b,m]t\otimes [b]s\Big)\label{Gamma_quat}.
\end{align}

We will prove that the $\alpha_i$'s and $\beta_j$'s are in fact real-valued by showing they are in the center of $\bb{H}$. 
There are two ways of doing this; one for 
the $\alpha_i$'s and another for the $\beta_j$'s.

\begin{lemma}
  The  $\alpha_i$'s are real-valued functions.
\end{lemma}

\begin{proof}
We give the proof for $\alpha_2$.  The arguments for $\alpha_1$, and $\alpha_3$ are similar.  Let $h\in\mathbb H$ and consider the diagram below

\[
    \hspace{9mm}\begin{tikzcd}
		m \ar[ddd,"h"]\ar[rrr,"{\alpha_2(a,b)}"] &&&m\ar[ddd,"h"]& \\ 
		  &(a\otimes m)\otimes b \ar[ul,"f_1"]\ar[r,"\alpha"]\ar[d,"{(id_a \otimes h)\otimes id_b}"] 
		  & a\otimes (m\otimes b)\ar[d,"{id_a \otimes (h\otimes id_b)}"]\ar[ur,"f_2"]  \\
	&(a\otimes m)\otimes b\ar[dl, "f_1"] \ar[r,"\alpha"] &a\otimes (m\otimes b) \ar[dr,"f_2"] \\
		m \ar[rrr,"{\alpha_2(a,b)}"] &&&m \\
	\end{tikzcd}
\]

\vspace{-11mm}
\begin{gather}
	f_1 = [m,b] \circ \big([a,m]\otimes\id_b\big)\,,\,\;f_2 = [a,m] \circ \big(\id_a\otimes[m,b]\big)\,.\nonumber
\end{gather}

By naturality of the associator, the middle square commutes.  The top and bottom quadrangles commute by the definition of $\alpha_2(a,b)$ and our choices of $[a,m]$ and $[m,b]$.  The quadrangles on the left and right commute by our choice of basis vectors $[a,m]$ and $[m,b]$.  It follows that the outer rectangle commutes.

Since $h\in\mathbb H$ was arbitrary, $\alpha_2(a,b)$ must lie in the center of $\mathbb H$, which is $\mathbb R$.  Since $a,b\in A$ were arbitrary, all values of $\alpha_2$ must be real numbers.

\end{proof}


\begin{lemma}
The  $\beta_j$'s are real-valued functions.    
\end{lemma}

\begin{proof}
We give the proof for $\beta_1$.  The arguments for $\beta_2$, and $\beta_3$ are similar.

Let $a,b\in A$ and $h\in\mathbb H$.  Naturality of the associator can be combined with Equation \ref{nice-basis} to show that
\begin{align*}
    [b]\big([a,m]\otimes\beta_1(a,b)h\big)&=
    [b]\big([a,m]\otimes h\beta_1(a,b)\big)\,.
\end{align*}

Since $m$ is self-dual and it is fixed by all elements in $A$, there is an isomorphism
\[\Hom\big((a\otimes m)\otimes m\,,\,b\big)\mathop{\longrightarrow}\limits^{\phi}\Hom(m,m)=\mathbb H\,.\]
This map $\phi$ is an isomorphism of right $\mathbb H$-modules. Now we define the quaternion
\[q:=\phi\Big([b]\big([a,m]\otimes \id_m\big)\Big)\,.\]
Since $\phi$ is an isomorphism, the morphism $q\in\mathbb H$ is nonzero and hence invertible.Then we have  that

\[\beta_1(a,b)h\;=\;q^{-1}q\beta_1(a,b)h\;=\;q^{-1}qh\beta_1(a,b)\;=\;h\beta_1(a,b)\,.\]

Thus $\beta_1(a,b)$ commutes with $h$.  Since $h$, $a$, and $b$ were arbitrary, the result follows.

\end{proof}



Finally, we will consider the nature of $\alpha_{m,m,m}$.
With $\End(\1)=\mathbb R$ and $\End(m)=\mathbb H$, the associator coefficient $\alpha_{m,m,m}$ is generically
a matrix in $GL_{16|A|}(\mathbb{R})$ (see Remark \ref{Rem:GammaMatrixDim}). In all
of the sums that follow, $a,b \in A$, $r,s,t,s',t' \in \mathbb H$, and $S =
\{1,i,j,k\}\subset\mathbb H$. For this computation, we fix the following notation:
\[
	\zeta(b,r,s,t) := [b,m]\left([b]\otimes \id_m\right)\big((r\otimes s)\otimes t\big)\,,\hspace{4mm}\text{and}
\]
\[
	{\mathcal{A}} (a):= [m,a](\id_m\otimes [a])\alpha_{m,m,m}\,.
\]
With this notation in hand, we can write
\begin{equation}\label{gamma_first}
	{\mathcal{A}}(a) = \sum_{\substack{b\in A,\\s,t\in S}} \gamma(a,b,s,t)\zeta(b,1,s,t).
\end{equation}
Here the coefficients $\gamma(a,b,s,t)$ are real.  We extend by $\mathbb R$-linearity in the $s$ and $t$ arguments so that, for example, $\gamma(a,b,-i,j):=-\gamma(a,b,i,j)$.  Furthermore, using naturality, we get that 
\begin{gather*}
	 \sum_{\substack{b\in A,\\s,t\in S}}\gamma(a,b,s,t)\zeta(b,1,s,t)\;=\;{\mathcal{A}}(a)\;=\;(r\inv r \circ {\mathcal{A}}(a))\\
     =\;r\inv \circ {\mathcal{A}}(a) \circ ((r\otimes \id_m )\otimes \id_m)\,
	\;=\;\sum_{\substack{b\in A,\\s',t'\in S}}\gamma(a,b,s',t')(r\inv \circ \zeta(b,r,s',t'))\\
    =\;\sum_{\substack{b\in A,\\s',t'\in S}}\gamma(a,b,s',t')\zeta\big(b,r,s',r\inv t'\big)\,
    \;=\;\sum_{\substack{b\in A,\\s',t'\in S}}\gamma(a,b,s',t')\zeta\big(b,1,\overline{r}s',r\inv t'\big)
\end{gather*}
Since these two sums must be equal, by equating the coefficients of the basis vectors we obtain the following relation
\begin{gather*}
    \gamma(a,b,s,t)\;=\;\gamma(a,b,\overline{r}^{-1}s,rt)\,.
\end{gather*}
When $r = \bar{s}$, then $s' = 1, t' = \bar{s}t$, and hence
$\gamma(a,b,s,t) = \gamma(a,b,1,\bar{s}t)$.  In particular, we find that $$\gamma(a,b,i,i) =
\gamma(a,b,j,j)=\gamma(a,b,k,k) = \gamma(a,b,1,1)\,.$$
A similar computation, this time involving $\mathcal A(a)\circ(\id_{m\otimes m}\otimes rr^{-1})$, implies that
\linebreak $\gamma(a,b,s,t) = \gamma(a,b,1,t\bar{s})$.  Thus we find that $\gamma(a,b,s,t)$ is simultaneously equal to $\gamma(a,b,1,\bar{s}t)$ and $\gamma(a,b,1,t\bar{s})$.  In particular, $\gamma(a,b,s,t) = 0$ if $t\neq s$.

We now simplify Equation \ref{gamma_first} using these observations to get
\begin{equation*}
	{\mathcal{A}}(a) = \sum_{b,s}  \gamma(a,b,1,1)\zeta(b,1,s,s)\,.
\end{equation*}
This reduction in complexity suggests that we set $\gamma(a,b):=\gamma(a,b,1,1)$ to finally arrive at
\begin{equation*}
	{\mathcal{A}}(a) = \sum_{b,s}  \gamma(a,b)\zeta(b,1,s,s).
\end{equation*}

\subsection{The Pentagon Equations}

With the associators simplified as much as possible, and with all but
$\alpha_{m,m,m}$ shown to be real-valued functions, now the pentagon equations themselves must be
analyzed. Using the rightmost factor convention (see Note \ref{Note:RightmostFactor}), we go through each of the 16 pentagons as was done in \cite{TAMBARA1998692}.  This results in
the following equations

\begin{align}
\delta\alpha&=1,\label{Eq:Quat1}\\
\delta\alpha_3&=\alpha^{-1},\label{Eq:Quat2}\\
\delta\alpha_1&=\alpha,\label{Eq:Quat3}\\
\alpha_2(a,bc)&=\alpha_2(a,c)\alpha_2(a,b),\label{Eq:Quat4}\\
\alpha_2(ab,c)&=\alpha_2(b,c)\alpha_2(a,c),\label{Eq:Quat5}\\
\alpha(a,b,b^{-1}a^{-1}c)\beta_1(ab,c)&=\beta_1(b,a^{-1}c)\beta_1(a,c)\alpha_3(a,b),\label{Eq:Quat6}\\
\beta_3(ab,c)\alpha(cb^{-1}a^{-1},a,b)&=\alpha_1(a,b)\beta_3(b,c)\beta_3(a,cb^{-1}),\label{Eq:Quat7}\\
\beta_2(b,c)&=\beta_2(b,a^{-1}c)\alpha_2(a,b),\label{Eq:Quat8}\\
\beta_2(a,c)&=\alpha_2(a,b)\beta_2(a,cb^{-1}),\label{Eq:Quat9}\\
\beta_1(a,c)\beta_3(b,c)&=\beta_3(b,a^{-1}c)\alpha(a,a^{-1}cb^{-1},b)\beta_1(a,cb^{-1}),\label{Eq:Quat10}\\
\beta_2(a,c)\beta_2(b,c)&=\alpha_3(a,b)\beta_2(ab,c)\alpha_1(a,b),\label{Eq:Quat11}\\
\alpha_2(a,c)\gamma(c,b) &= \beta_1(a,b)\alpha_3(a,a^{-1}b)\gamma(c,a^{-1}b),\label{Eq:Quat12}\\
\alpha_2(b,a)\gamma(c,b) &= \beta_3(a,c)\alpha_1(ca^{-1},a)\gamma(ca^{-1},b),\label{Eq:Quat13}\\
\alpha_1(a,c)\gamma(c,b) &= \beta_2(a,b)\beta_1(a,ac)\gamma(ca,b),\label{Eq:Quat14}\\
\alpha_3(b,a)\gamma(c,b) &= \beta_2(a,c)\beta_3(a,ba)\gamma(c,ba),\label{Eq:Quat15}\\
\delta_{d,ba^{-1}}\beta_3(a,b)\beta_1(ba^{-1},b)&=4\sum_{c}\beta_2(c,b)\gamma(c,d)\gamma(a,c).\label{Eq:Quat16}
\end{align}

After some close inspection, it is clear that these are the same 16 equations
that were obtained 
for the pentagons in  \cite[page 699]{TAMBARA1998692}, with the
exception of the last pentagon having a four on the right-hand side.

\subsection{Rescaling}\label{sub:QuatRescale}

Because all of the associator coefficients, except $\alpha_{m,m,m}$, are real-valued, the algebraic manipulations by which Tambara and Yamagami derived all of the associators information in \cite{TAMBARA1998692} also works in this case.
However, there is one small change. 
 The equation
below appearing in \cite{TAMBARA1998692} 
\[\gamma(1,1)^2\sum_{c\in A} \alpha_2(a,bd^{-1}a^{-1}) =\delta_{b,ad},
\]
needs to have a four on the left-hand side in the quaternionic case
$$4\gamma(1,1)^2\sum_{c\in A} \alpha_2(a,bd^{-1}a^{-1}) =\delta_{b,ad}.$$
After normalization, the associators become 
\[\alpha \equiv 1\,,\quad \alpha_1 = \alpha_3 = \beta_1 = \beta_3 \equiv 1\,,\quad \alpha_2=\beta_2,\quad\gamma(a,b) = \frac{\gamma(1,1)}{\alpha_2(a,b)},
\]
where $\alpha_2$ is a nondegenerate symmetric bicharacter, and 
where $4\gamma(1,1)^2|A| =1$. Therefore all of the associators are uniquely 
determined by $\alpha_2$ and $\gamma(1,1)$.
In conclusion, we have the following result 
\begin{theorem}\label{Thm:TYQuaternionic}
    Let $A$ be a finite group, let $\tau=\sfrac{\pm1}{\sqrt{4|A|}}$, and let $\chi:A\times A\to \mathbb R^\times$ be a nongedegerate symmetric bicharacter on $A$.
    
    A triple of such data gives rise to a non-split Tambara-Yamagami category \newline $\s C_{\bb H}(A,\chi,\tau)$, with $\End(\1)\cong\bb R$ and $\End(m)\cong\bb H$.  Furthermore, all equivalence classes of such categories arise in this way.  Two categories $\s C_{\bb H}(A,\chi,\tau)$ and $\s C_{\bb H}(A',\chi',\tau')$ are equivalent if and only if $\tau=\tau'$ and there exists an isomorphism $f:A\to A'$ such that for all $a,b\in A$,
    \[\chi'\big(f(a),f(b)\big)\;=\;\chi(a,b)\,.\]
\end{theorem}

\begin{proof}
    We must establish necessary and sufficient conditions for the existence of $\mathcal C_{\mathbb H}(A,\chi,\tau)$, and then establish necessary and sufficient conditions for the existence of a monoidal equivalence $(F,J):\mathcal C_{\mathbb H}(A,\chi,\tau)\to\mathcal C_{\mathbb H}(A',\chi',\tau')$.
    The analysis leading up to the theorem establishes the first necessity statement, and so we begin by showing that our conditions are sufficient for the existence of $\mathcal C_{\mathbb H}(A,\chi,\tau)$.

    Since the simple objects, endomorphism algebras, and fusion rules are already prescribed, we only need to write down the associators and prove that they are coherent. In order to write down certain associators, we will use the construction from Proposition \ref{Prop:RealPartWDagger} to assume without loss of generality that the category has projections $[a]:m\otimes m\to a$ and inclusions $[a]^\dagger:a\to m\otimes m$ such that the following equations hold for every $a\in A$, and $h\in\mathbb H$,
    \begin{gather}
        [a](\id_m\otimes h)[a]^\dagger\;=\;\Re(h)\cdot\id_a\,,\text{ and}\label{Eq:IncludeProjectReal}\\
        \id_{m\otimes m}\;=\;\sum_{a\in A, s\in S}(\id_m\otimes\overline{s})[a]^{\dagger}[a](\id_m\otimes s)\,.
    \end{gather}
    
    The associators of $\mathcal C_{\mathbb H}(A,\chi,\tau)$ are given, for $a, b,
	c\in A$, as follows:
	\begin{gather*}
	    \alpha_{a,b,c}=\id_{abc}\,,\\
	    \alpha_{a,b,m}=\alpha_{m,b,c}=\id_{m}\,,\\
	    \alpha_{a,m,c}=\chi(a,c)\cdot\id_{m},\\
	    \alpha_{a,m,m}=\alpha_{m,m,c}=\id_{m\otimes m}\,,\\
	    \alpha_{m,b,m}=\bigoplus_{a\in A}\chi(a,b)\cdot\id_{a^{\oplus4}}\,,\\
	    \alpha_{m,m,m}=\tau\cdot\sum_{\substack{a,b\in A\\s,t\in S}}\chi(a,b)^{-1}\cdot(s\otimes(\id_m\otimes\overline{t}))(\id_m\otimes[a]^\dagger)([b]\otimes\id_m)((\id_m\otimes s)\otimes t).
	\end{gather*}
	In this last equation we have used the fact that $b\otimes m=m=m\otimes a$.  The unit is $\1=\mathbf 1_A$, the identity in $A$, and the unit constraints are identity morphisms.
    
    By plugging in $\alpha_2=\chi=\beta_2$, $\gamma(a,b)=\tau\cdot\chi(a,b)^{-1}$ and all others constant with value $1$, Equations \ref{Eq:Quat1}-\ref{Eq:Quat16} are satisfied.  The fact that $\chi$ is a bicharacter proves that Equations \ref{Eq:Quat4}, \ref{Eq:Quat5}, and \ref{Eq:Quat10} hold.  The fact that $\chi$ is symmetric proves that Equations \ref{Eq:Quat8} and \ref{Eq:Quat9} hold. The nondegeneracy of $\chi$ and the fact that $\tau^2\cdot4|A|=1$ together imply Equation \ref{Eq:Quat16} is true.  All the remaining pentagon equations follow immediately from the definitions.
    
    Next, suppose that there is a monoidal equivalence $$(F,J):\mathcal{C}_{\mathbb H}(A,\chi,\tau)\to\mathcal{C}_{\mathbb H}(A',\chi',\tau')\,.$$  Monoidal equivalences send invertible objects to invertible objects, and so $F$ must act by some group isomorphism $f:A\to A'$.  Since $m'$ is the only quaternionic simple object in $\mathcal{C}_{\mathbb H}(A',\chi',\tau')$, we must have that $m'\cong F(m)$, and hence $\Hom(m',F(m))\neq0$ is a simple $(\mathbb H,\mathbb H)$-bimodule.  For any $v\in\Hom(m',F(m))$ and $h\in\mathbb H$, the bimodule structure is given by
    \[h.v\;=\;F(h)\circ v\,,\qquad\text{ and }\qquad v.h\;=\;v\circ h\,.\]
    Since $\Hom(m',F(m))$ is a simple bimodule, Proposition \ref{Prop:basis} shows that there is some nonzero $y\in\Hom(m',F(m))$ such that $h.y=y.h$.  This formula is equivalent to $F(h)=y\circ h\circ y^{-1}$.  The components of the tensorator have four different types: $J_{a,b}$, $J_{a,m}$, $J_{m,b}$ and $J_{m,m}$, and we can use the isomorphism $y:m'\to F(m)$ to extract them as follows (composition symbols are omitted for space):
    \begin{align}
        F\big([a,b]\big)J_{a,b}&=\theta(a,b)^{-1}\cdot\big[f(a),f(b)\big],\\
        F\big([a,m]\big)J_{a,m}(\id_{f(a)}\otimes y)&=y\varphi(a)^{-1}\big[f(a),m\big],\\
        F\big([m,b]\big)J_{m,b}(y\otimes\id_{f(b)})&=y\psi(b)^{-1}\big[m,f(b)\big],\\
        F\big([a]\big)J_{m,m}(y\otimes y)&=\big[f(a)\big]\big(\id_m\otimes\omega(a)^{-1}\big).
    \end{align}
    The inverses here are simply a convention as explained in Remark \ref{Rem:InversesInTheTensorators}. 

    Just as the naturality of the associator implied that the associator coefficients were real-valued, the naturality of $J$ implies that the tensorator coefficients $\theta, \varphi, \psi$, and $\omega$ are also all real-valued.  The hexagon relations for the tensorators produce the following equations:
    \begin{align}
        1&=\delta\theta\label{QuatEquivalence1}\,,\\
        \theta&=\delta(\psi)\label{QuatEquivalence2}\,,\\
        \chi'\Big(f(a),f(b)\Big)&=\chi(a,b)\label{QuatEquivalence3}\,,\\
        \theta&=\delta(\varphi)\label{QuatEquivalence4}\,,\\
        \varphi(a)\omega(b)&=\omega(a^{-1}b)\theta(a,a^{-1}b)\label{QuatEquivalence5}\,,\\
        \chi'\Big(f(a),f(b)\Big)&=\frac{\varphi(a)}{\psi(a)}\cdot\chi(a,b)\label{QuatEquivalence6}\,,\\
        \theta(ba^{-1},a)\omega(ba^{-1})&=\psi(a)\omega(b)\,,\label{QuatEquivalence7}\\
        \frac{\tau}{\chi(a,b)\varphi(b)\omega(b)}&=\frac{\tau'}{\chi'\big(f(a),f(b)\big)\psi(a)\omega(a)}\label{QuatEquivalence8}\,.
    \end{align}
    Equation \ref{QuatEquivalence1} is implied by Equations \ref{QuatEquivalence2} and \ref{QuatEquivalence4}.  Equations \ref{QuatEquivalence3} and \ref{QuatEquivalence6} imply that $\varphi=\psi$, which makes Equations \ref{QuatEquivalence2} and \ref{QuatEquivalence4} equivalent to one another.  Equation \ref{QuatEquivalence3} can be used to reduce Equation \ref{QuatEquivalence8} to
    \begin{align}
        \frac{\tau}{\varphi(b)\omega(b)}&=\frac{\tau'}{\psi(a)\omega(a)}\,.
    \end{align}
    Since the left-hand side only depends on $b$, and the right-hand side only depends on $a$, this quantity must depend on neither $a$ nor $b$.  Setting $a=b$ we find that $\tau=\tau'$.  Thus the existence of a monoidal equivalence implies the desired relations.

    By removing redundancies, the equations above reduce to the following list:
    \begin{align}
        \theta&=\delta(\psi)\label{QuatEquivalenceReduced1},\\
        \chi'\Big(f(a),f(b)\Big)&=\chi(a,b)\label{QuatEquivalenceReduced2},\\
        \psi&=\varphi\label{QuatEquivalenceReduced3},\\
        \tau&=\tau'\label{QuatEquivalenceReduced4},\\
        \varphi(a)\omega(a)&=\varphi(1)\omega(1)\label{QuatEquivalenceReduced5}\,.
    \end{align}

    Finally, suppose that $\chi'\big(f(a),f(b)\big)=\chi(a,b)$, and $\tau=\tau'$.  We can construct a tensorator $J$ by writing down coefficient functions $\theta, \varphi, \psi$, and $\omega$.  The coherence of $J$ is then equivalent to the validity of Equations \ref{QuatEquivalenceReduced1}-\ref{QuatEquivalenceReduced5}.  By setting all of these functions to be constant with value $1$, the coherence of the resulting $J$ is immediate.  Thus, these relations between $\chi'$ and $\chi$, $\tau$ and $\tau'$ are enough to prove the existence of a monoidal equivalence between the two categories, and the proof is complete.
\end{proof}

\begin{example}\label{Eg:SimpleQuatCat}
The simplest example of such a category is $\mathcal C_{\mathbb H}(\mathbf 1,1,\pm\sfrac{1}{2})$. The simple objects are
${\1}$ and $m$. By construction $\End({\1}) \cong \mathbb{R}$ and
$\End(m) \cong \mathbb{H}$, with the only non trivial fusion rule being $m\otimes m = 4\cdot\1$.
Since there are no non-trivial group automorphisms and no non-trivial
bicharacters for the trivial group, there are only two categories arising from
this group over $\mathbb{R}$, one for $\sfrac12$ and another for $-\sfrac12$. In the
notation, $1$ stands for the trivial bicharacter from the trivial group to
$\mathbb{R}$. There is only one associator which is
non-trivial, $\alpha_{m,m,m}$. Since $\tau = \pm \sfrac{1}{2}$ and $\chi$ is
always trivial, this means that the following equation completely describes the associator:
\begin{gather*}
	[m,1](\id_m \otimes [1])\alpha_{m,m,m} = \sum_{b,s}  \gamma(a,b)\zeta(b,1,s,s) = 
	\pm \frac{1}{\;2\;}\cdot\sum_{s \in \{1,i,j,k\}}\zeta(1,1,s,s)\,.
\end{gather*}
As can be seen above, there are only two possible choices for the only
non-trivial associator $\alpha_{m,m,m}$, and this comes down to a choice of sign.
These categories are not new; $\mathcal C_{\mathbb H}(\mathbf 1,1,\sfrac12)$ and $\mathcal
C_{\mathbb H}(\mathbf 1,1,-\sfrac12)$ were described in \cite{etingofDescentAndForms}, where
they arose as examples of real forms of
$\text{Vec}_\mathbb{R}^\omega(\mathbb{Z}/2\mathbb Z)$ for $\omega = 0$ and $1$,
respectively.  The category $\mathcal C_{\mathbb H}(\mathbf 1,1,\sfrac12)$ has also appeared in \cite{theoSpinStatistics} where it was given a symmetric braiding, referred to as $\mathsf{SuperVect}_{\mathbb H}$, and interpreted as a categorified field extension of $\Vec_{\mathbb R}$.
\end{example}

\begin{example}\label{Eg:NewQuatCat}
    Let $A=\mathbb Z/2\mathbb Z=\langle w\rangle$, and set $\chi(w,w)=-1$.  The object $m$ in $\mathcal C_{\mathbb H}\left(A,\chi,\pm\tfrac{1}{2\sqrt2}\right)$ satisfies $m\otimes m\cong 4\cdot(\1\oplus w)$.
    From this, it follows that $\FPdim(m)=2\sqrt2$.  This implies that the categories $\mathcal C_{\mathbb H}\left(A,\chi,\pm\tfrac{1}{2\sqrt2}\right)$ do not even admit quasi-fiber functors, and thus cannot be realized as $\Rep_{\mathbb R}(H)$ for any quasi-Hopf algebra $H$ over $\mathbb R$.
\end{example}

\begin{remark}
    Theorem \ref{Thm:TYQuaternionic} requires $\chi:A\times A\to\mathbb R^\times$ to be nondegenerate.  The only groups for which this is possible are elementary abelian 2-groups, that is, groups of the form $(\mathbb Z/2\mathbb Z)^n$.  In this sense, Examples \ref{Eg:NewQuatCat} and \ref{Eg:SimpleQuatCat} are generic.
\end{remark}

\begin{proposition}\label{Prop:QuatCatsAreRigid}
    The categories $\mathcal C_{\mathbb H}(A,\chi,\tau)$ are rigid.
\end{proposition}

\begin{proof}
    It will suffice to show that all simple objects have duals.  Since invertible objects are always dualizable, the only object we need to check is $m$.  We choose $[1]:m\otimes m\to \1$ to be the evaluation map, and $\tau^{-1}[1]^\dagger:\1\to m\otimes m$ to be the coevaluation map.  A short computation shows that the morphism $\alpha_{m,m,m}^{-1}$ is given by
    \[\alpha_{m,m,m}^{-1}=\tau\cdot\sum_{\substack{a,b\in A\\s,t\in S}}\chi(a,b)\big((\id_m\otimes\overline{s})\otimes\overline{t}\big)([b]^\dagger\otimes\id_m)(\id_m\otimes[a])\big(\overline{s}\otimes(\id_m\otimes t)\big)\,.\]
    The duality equations follow from these formulas, Equation \ref{Eq:IncludeProjectReal}, and Schur's Lemma (Lemma \ref{SchursLemma}).
\end{proof}

\section{Analysis of the Real-Complex Case\label{Sec:Real/Complex}}

In this section, we will construct the non-split Tambara-Yamagami categories \newline $\mathcal C_{\mathbb C}(G,g,\chi,\tau)$, where $\1$ is real, and $m$ is complex. Each of the spaces
\begin{gather*}
     \Hom(a\otimes m,m)\,, \hspace{1cm} \Hom(m\otimes a,m)\,, \hspace{.3cm}\text{ and }\hspace{.5cm} \Hom(m\otimes m,a)
\end{gather*}
are 1-dimensional complex bimodules.  Every such bimodule is isomorphic to either
the trivial bimodule $\bb C$ or the conjugating bimodule $\overline{\mathbb C}$, in which the left
and right actions differ by conjugation.
\begin{definition}\label{Def:The-a-symbol}
    For an element $a\in G$ and a scalar $\lambda\in \bb C$, define the superscript notation
    \[\lambda^a:=\begin{cases}
        \lambda & \text{ if }\Hom(a\otimes m,m)\cong\bb C\,,\\
        \overline{\lambda}& \text{ if }\Hom(a\otimes m,m)\cong\overline{\mathbb C}\,,\\
    \end{cases}\]
    and also the degree
    \[|a|:=\begin{cases}
        0 & \text{ if }\Hom(a\otimes m,m)\cong\bb C\,,\\
        1 & \text{ if }\Hom(a\otimes m,m)\cong\overline{\mathbb C}.\\
    \end{cases}\]
         We say that $a$ \emph{conjugates} when $|a|=1$. 
\end{definition}

The appearance of conjugating bimodules should not be surprising.  In
\cite{etingofFusionCategoriesHomotopy2009}, Etingof, Nikshych, and Ostrik analyze
Tambara-Yamagami type fusion categories as $\bb Z/2\bb Z$-graded extensions of
pointed categories in the algebraically closed setting, which is split.  Using their language the categories we are considering in this section are still $\mathbb Z/2\mathbb Z$-graded, and hence $\s M:=\mathbb C\text{-}\Vec=\langle
m\rangle$ would necessarily be an invertible bimodule category for the
pointed category $\s C:=\mathbb R\text{-}\Vec^\omega_{G}$.  If none of the $a\in G$ acted
by the conjugation functor, then all objects in the dual category $\s C_{\s
M}^*$ would be complex, and this would imply that $\s M$ wasn't invertible, because invertibility forces $\mathcal C\simeq\mathcal C_{\mathcal M}^*$.

Thus we find that there must be at least one element of $G$ that conjugates.  On the
level of groups, the degree map defined above must be a surjective homomorphism onto $\bb
Z/2\bb Z$.  Let $A$ be the kernel of the degree map, so that we have a short
exact sequence: \[A\hookrightarrow G\twoheadrightarrow\bb Z/2\bb Z\,.\] 

We will uncover more about the structure of $G$ in Lemma \ref{GeneralizedDihedral}, but this will require further information in the form of the pentagon equations.  For now, we point out that $|a|=|a^{-1}|$, for all $a\in G$. This allows us to replace expressions like $\lambda^{a^{-1}}$ with $\lambda^{a}$ in order to avoid nested superscripts.


We will choose arbitrary nonzero morphisms for each $a, b$ and $c\in G$,
\begin{gather*}
     [a,b]\in\Hom(a\otimes b, c)\,,\qquad
     [a,m]\in\Hom(a\otimes m,m)\,,\\
     [m,a]\in\Hom(m\otimes a,m)\,,\quad\text{and}\qquad
     [a]\in\Hom(m\otimes m,a)\,.
\end{gather*}

A priori, there is nothing to indicate which type of bimodule $\Hom(m\otimes m,\1)$ happens to be.  For now, we will keep track of this possible conjugation by the symbol $g$. More explicitly, we have the following rule.
\begin{definition}\label{Def:The-g-Symbol}
    There is an $\bb R$-linear automorphism of $\bb C$, denoted $\lambda\mapsto\lambda^g$, that is uniquely determined by the equation
    \[[1]\circ(\id_m\otimes i)\;=\;[1]\circ(i^g\otimes\id_m)\,.\]
\end{definition}





Once again following the notation of \cite{TAMBARA1998692}, below are the associators of the category for $a,b,c\in G$
\begin{align*}
	\tetr{a}{b}{c}{abc}:& [b,c]\otimes[a,bc] \to \alpha(a,b,c)\cdot[a,b]\otimes [ab,c]\,,\\
	\tetr{a}{b}{m}{m}:&   [b,m]\otimes [a,b] \to [a,b]\otimes\Big([ab,m]\alpha_3(a,b)\Big)\,,\\
	\tetr{a}{m}{b}{m}:&   [m,b]\otimes[a,m]  \to \Big([a,m]\alpha_2(a,b)\Big)\otimes[m,b]\,,\\
	\tetr{m}{a}{b}{m}:&   [a,b]\otimes[m,ab] \to \Big([m,a]\alpha_1(a,b)\Big)\otimes [m,b]\,,\\
	\tetr{a}{m}{m}{b}:&   [a^{-1}b]\otimes [a,a^{-1}b] \to [a,m]\otimes \Big([b]\beta_1(a,b)\Big)\,,\\
	\tetr{m}{a}{m}{b}:&   [a,m]\otimes [b] \to [m,a]\otimes \Big([b]\beta_2(a,b)\Big)\,,\\
	\tetr{m}{m}{a}{b}:&   [m,a]\otimes [b] \to \Big([ba^{-1}]\beta_3(a,b)\Big)\otimes [ba^{-1},a]\,,\\
	\tetr{m}{m}{m}{m}:&   [a]\otimes[m,a]\to \sum_{\substack{b\in A\\s\in\{1,i\}}} [b]s\otimes [b,m]\gamma(a,b)_s\label{Gamma0}\,.
\end{align*}


In this case, only $\alpha$ is real-valued.  The $\alpha_j$'s and $\beta_k$'s are complex-valued, and $\gamma$ takes values in $M_{4|G|}(\bb R)$ (see Remark \ref{Rem:GammaMatrixDim}).  By naturality, $\gamma$ will be determined by vectors $\gamma(a,b)$ in the algebra $\bb C\otimes_{\bb R}\bb C$ for each pair $a,b\in A$.  We use a notation that keeps track of the complex factor on the right by using the following decomposition
\[\gamma(a,b)\;=\;1\otimes\gamma(a,b)_1+i\otimes\gamma(a,b)_i\,.\]
Any element of $\bb C\otimes_{\bb R}\bb C$ can be given such a decomposition.

With these conventions in place, we can derive the following pentagon equations by applying the rightmost factor convention, see Note \ref{Note:RightmostFactor}.

\begin{align}
\delta\alpha&=1\label{Complex1}\,,\\
\delta^R\alpha_3&=\alpha^{-1}\label{Complex2}\,,\\
\delta^L\alpha_1&=\alpha\label{Complex3}\,,\\
\alpha_2(a,bc)\alpha_1(b,c)^a&=\alpha_1(b,c)\alpha_2(a,c)^b\alpha_2(a,b)\label{Complex4}\,,\\
\alpha_3(a,b)^c\alpha_2(ab,c)&=\alpha_2(b,c)\alpha_2(a,c)^b\alpha_3(a,b)\label{Complex5}\,,\\
\alpha(a,b,b^{-1}a^{-1}c)\beta_1(ab,c)&=\beta_1(b,a^{-1}c)\beta_1(a,c)\alpha_3(a,b)^{gabc}\label{Complex6}\,,\\
\beta_3(ab,c)\alpha(cb^{-1}a^{-1},a,b)&=\alpha_1(a,b)\beta_3(b,c)^a\beta_3(a,cb^{-1})\label{Complex7}\,,\\
\beta_2(a,c)^b\beta_3(b,c)&=\alpha_2(a,b)\beta_3(b,c)^a\beta_2(a,cb^{-1})\label{Complex8}\,,\\
\beta_1(a,c)^b\beta_2(b,c)&=\beta_2(b,a^{-1}c)\beta_1(a,c)\alpha_2(a,b)^{gabc}\label{Complex9}\,,\\
\beta_1(a,c)^b\beta_3(b,c)&=\beta_3(b,a^{-1}c)\alpha(a,a^{-1}cb^{-1},b)\beta_1(a,cb^{-1})\label{Complex10}\,,\\
\beta_2(a,c)^b\beta_2(b,c)&=\alpha_3(a,b)\beta_2(ab,c)\alpha_1(a,b)^{gabc}\label{Complex11}\,,\\
\sum_{s}s\cdot\alpha_2(a,c)^{gab}\otimes\gamma(c,b)_s&=\sum_tt\cdot\beta_1(a,b)\otimes\alpha_3(a,a^{-1}b)\cdot\gamma(c,a^{-1}b)_t\label{Complex12}\,,\\
\sum_ss\otimes\alpha_2(b,a)\cdot\gamma(c,b)_s^a&=\sum_tt\cdot\alpha_1(ca^{-1},a)^{gb}\otimes\beta_3(a,c)\cdot\gamma(ca^{-1},b)_t\label{Complex13}\,,\\
\sum_ss\cdot\alpha_1(a,c)^{gab}\otimes\gamma(c,b)_s&=\sum_tt^a\cdot\beta_2(a,b)\otimes\beta_1(a,ac)\cdot\gamma(ac,b)_t\label{Complex14}\,,\\
\sum_ss\otimes\alpha_3(b,a)\cdot\gamma(c,b)_s^a&=\sum_tt^a\cdot\beta_3(a,ba)\otimes\gamma(c,ba)_t\cdot\beta_2(a,c)\label{Complex15},\\
\delta_{d,ba^{-1}}\beta_3(a,b)\otimes\beta_1(ba^{-1},b)&=\sum_{c,s,t}t\otimes s^{gbd}\cdot\beta_2(c,b)\cdot\gamma(c,d)^{gbd}_t\cdot\gamma(a,c)_s\;.\label{Complex16}
\end{align}

Any change in the basis vectors can be achieved by a transformation of the following form
\begin{align*}
    [a,b]'&=\theta(a,b)[a,b]\,,\\
    [a,m]'&=\varphi(a)[a,m]\,,\\
    [m,a]'&=\psi(a)[m,a]\,,\\
    [a]'&=\omega(a)[a]\,,
\end{align*}
where $\varphi$, $\psi$ and $\omega$ are complex-valued functions, and $\theta$ is real-valued. Under such a change of basis, the associator coefficients change in the following way
\begin{align}
    \alpha'&=\delta\theta\cdot\alpha\label{Basis1}\,,\\
    \alpha_1'(a,b)&=\frac{\psi(ab)\theta(a,b)}{\psi(a)\psi(b)^a}\cdot\alpha_1(a,b)\label{Basis2}\,,\\
    \alpha_2'(a,b)&=\frac{\psi(b)\varphi(a)^b}{\psi(b)^a\varphi(a)}\cdot\alpha_2(a,b)\label{Basis3}\,,\\
    \alpha_3'(a,b)&=\frac{\varphi(b)\varphi(a)^b}{\varphi(ab)\theta(a,b)}\cdot\alpha_3(a,b)\label{Basis4}\,,\\
    \beta_1'(a,b)&=\frac{\omega(a^{-1}b)\theta(a,a^{-1}b)}{\varphi(a)^{gab}\omega(b)}\cdot\beta_1(a,b)\label{Basis5}\,,\\
    \beta_2'(a,b)&=\frac{\omega(b)^a\varphi(a)}{\omega(b)\psi(a)^{gab}}\cdot\beta_2(a,b)\label{Basis6}\,,\\
    \beta_3'(a,b)&=\frac{\psi(a)\omega(b)^a}{\theta(ba^{-1},a)\omega(ba^{-1})}\cdot\beta_3(a,b)\label{Basis7}.
\end{align}

The function $\gamma$ also changes, but we will not need to consider this until the proof of Theorem \ref{Thm:TYComplex}.  Unlike the split and quaternionic cases, Equation \ref{Basis1} does not allow us to trivialize $\alpha$.  We would set $\theta=\alpha_1^{-1}$ in those cases but in the current case $\alpha_1$ may take on complex values, while $\theta$ is real valued. Despite this, Equation \ref{Complex3} implies that $|\alpha|=\delta|\alpha_1|$.  Thus, by setting
\[\theta(a,b):=\frac{1}{|\alpha_1(a,b)|}\,,\]
we may assume that $|\alpha|\equiv|\alpha_1|\equiv1$.  Notice that we are still able to use $\{\pm1\}$-valued $\theta$'s to adjust the sign of $\alpha$ without altering the magnitudes.

Observe that Equation \ref{Complex4} can be rearranged to look like
\begin{align*}
    \frac{\alpha_2(a,bc)}{\alpha_2(a,c)^b\alpha_2(a,b)}\;&=\;\frac{\alpha_1(b,c)}{\alpha_1(b,c)^a}\,.
\end{align*}
We fix an $a$ that conjugates and define $f(x):=\alpha_2(a,x)^{-1}$.  Under these conditions, the above equation becomes $\delta^L f\;=\;\alpha_1^2$.
We now choose a function $\psi:G\to\bb C^\times$ that satisfies $\psi^2=f$.  It follows that
\begin{align*}
    \theta(x,y):=\frac{(\delta^L\psi)(x,y)}{\alpha_1(x,y)}\in\{\pm1\}.
\end{align*}
With this choice of $\psi$ and $\theta$, Equation \ref{Basis2} shows that we may assume $\alpha_1\equiv1$. It follows from Equation \ref{Complex3} that this also forces $\alpha\equiv1$.

By the triangle axiom, we may assume the following normalization conditions.  For each of $\alpha_3$, $\beta_1$, and $\beta_3$, either input being $1$ implies the output is $1$.  For $\beta_2$, we have $\beta_2(1,-)\equiv1$.  Notice that the triangle axiom gives no information about $\beta_2(-,1)$.

We further normalize by setting
\[\varphi(a):=\frac{\omega(1)\psi(a)^{ga}}{\omega(1)^a\beta_2(a,1)}\,.\]
By Equation \ref{Basis6}, this normalization allows us to assume  $\beta_2(-,1)\equiv1$. Thus we have shown that any Tambara-Yamagami category of this form has a basis for the hom spaces for which the coefficients of the associator are normalized. From now on, without loss of generality, we assume all coefficients to be normalized.

By normalization, setting $a=1$ or $b=1$ in Equations \ref{Basis1}-\ref{Basis7} shows that only certain basis changes are allowed.  The new restrictions are:
\begin{align}
    \begin{split}
    \psi(1)\;&=\;\theta(a,1)\;=\;\theta(1,b)\,,\\
    \varphi(a)&=\frac{\psi(a)^{ga}\omega(1)}{\omega(1)^a}\,,\\
    \varphi(a)^{ga}\omega(1)&=\omega(a^{-1})\theta(a,a^{-1})\,,\\
    \theta(a^{-1},a)\omega(a^{-1})&=\psi(a)\omega(1)^a\,.
    \end{split}\label{NormCoBTRules}
\end{align}
This shows that the quadruple $(\theta,\psi,\varphi,\omega)$ is completely determined by the triple $\big(\theta,\psi,\omega(1)\big)$.  The above equations imply that
\[\frac{\omega(1)^a}{\omega(1)}\;=\;\left(\frac{\omega(1)^a}{\omega(1)}\right)^g,\hspace{6mm}\text{ for any }a\in G\,.\]
Notice that when $g$ conjugates, this forces $\omega(1)^4=1$.

When comparing categories with the same fusion rules, if their associator coefficients differ by a change of bases then they must be equivalent.  When two categories are equivalent, they are often equivalent in multiple different ways. This is inconvenient for classification since it means there are many variables to check. Luckily, there is a way to normalize the equivalences by composing them with autoequivalences. In this setting, an autoequivalence means any change of bases that does not alter any of the associator coefficients.

Notice that the transformation resulting from using constant scalar factors \newline $\big(\theta,\psi,\omega(1)\big)=(r_1,r_1,r_2)$, for $r_1, r_2\in\bb R^\times$, sends all associator coefficients to themselves.  In other words, this corresponds to an autoequivalence of the category. By composing an equivalence $\big(\theta,\psi,\omega(1)\big)$ with an autoequivalence $\big(\psi(1)^{-1},\psi(1)^{-1},$\newline $|\omega(1)|^{-1}\big)$, we may assume that $\theta(a,1)=\theta(1,b)=\psi(1)=1$, and that $|\omega(1)|=1$. If $\lambda^4=1$, then the transformation $\big(\theta,\psi,\omega(1)\big)=(1,1,\lambda)$ also sends all the coefficients to themselves.  Thus we can freely replace $\omega(1)$ with $i^k\cdot\omega(1)$ by composing with an autoequivalence of the form $(1,1,i^k)$.  When $g$ conjugates, this means that we may assume $\omega(1)=1$, but when $g$ doesn't conjugate, we cannot make this assumption. 

\hspace{2mm}



With our normalization assumptions in place, we can continue to determine the dependencies between the coefficients.
To start, we can set $c=1$ in Equation \ref{Complex11} to find that $\alpha_3=1$.  Next, we set $c=1$ in Equation \ref{Complex6} to find that $\beta_1=1$.  By setting $c=1$ in Equation \ref{Complex7}, we see that $\beta_3\equiv1$.

Let us turn our attention now to $\alpha_2$ and $\beta_2$.  By setting $c=1$ in Equation \ref{Complex8}, we get that
\begin{align}
    \beta_2(a,b)=\alpha_2(a,b^{-1})^{-1}\label{Beta2}\,.
\end{align}
By setting $c=1$ in Equation \ref{Complex9}, and then applying Equation \ref{Beta2}, we obtain the following symmetry condition for $\alpha_2$
\begin{align}
    \alpha_2(b,a)=\alpha_2(a,b)^{gab}\label{Symmetry}\,.
\end{align}

Equation \ref{Complex4} is a cocycle condition, which can be rearranged to express a kind of twisted multiplicativity statement
\begin{align}
\alpha_2(a,bc)&=\alpha_2(a,b)\alpha_2(a,c)^b\,.\label{Multiplicativity}
\end{align}
In terms of cohomology, this kind of multiplicativity twisted by a group action is called a 2-cocycle condition.

\begin{remark}
    When restricted to $A\times A$, Equations \ref{Symmetry} and \ref{Multiplicativity} show that $\alpha_2$ is a complex-valued bicharacter.  If $g$ conjugates, then $\alpha_2$ is conjugate-symmetric (one might call it hermitian), and if $g$ doesn't conjugate, then $\alpha_2$ is symmetric.
\end{remark}

\begin{remark} Here is an interpretation of this multiplicativity.  There is a real algebra $\bb C_\sim G$ defined similarly to the standard group algebra $\bb CG$, except that group elements only commute with scalars up to conjugation:
    \[b\cdot\lambda\;=\;\lambda^b\cdot b.\]
    Equation \ref{Multiplicativity} can be understood as saying that for each $a\in G$, the map
    \[\alpha_2(a,-):\bb C_\sim G\to\End_{\bb R}(\bb C)\]
    endows $\bb C$ with the structure of a representation.
\end{remark}

Equation \ref{Complex5} is similar to Equation \ref{Complex4}, and the corresponding multiplicativity statement is a necessary consequence of Equations \ref{Symmetry} and \ref{Multiplicativity}.

Let us now turn our attention to the problem of simplifying the function $\gamma$.

\begin{lemma}\label{gamma(1,1)1vsiLemma} The entry $\gamma (1,1)$ satisfies the relation $\gamma(1,1)_i\;=\;\overline{i}^g\cdot\gamma(1,1)_1$.
\end{lemma}
\begin{proof}
    There are two summands of $m\otimes m$ that correspond to $\1$, the component corresponding to $1$ and the component corresponding to $i$.  By composing with the adjoint $[1]^\dagger$, we can isolate the coefficient $\gamma(1,1)_i$ as follows
    \begin{align*}
        \gamma(1,1)_i&=(\id_m\otimes[1])\circ\alpha_{m,m,m}\circ(\id_m\otimes(\overline{i})\otimes\id_m)\circ([1]^\dagger\otimes\id_m)\\
        &=(\overline{i}^g)\circ(\id_m\otimes[1])\circ\alpha_{m,m,m}\circ([1]^\dagger\otimes\id_m)\\
        &=\overline{i}^g\gamma(1,1)_1\,.
    \end{align*}
\end{proof}

In Equation \ref{Complex14}, by setting $b=c=1$, we find that
\begin{align}
    \sum_ss\otimes\gamma(1,1)_s&=\sum_tt^a\otimes\cdot\gamma(a,1)_t\,,\nonumber
\end{align}
which implies that
\begin{gather}
    \gamma(a,1)_s\;=\;(s^2)^{|a|}\gamma(1,1)_s
    \;=\;\overline{s}^{ga}\gamma(1,1)_1\label{GammaLeft}\,.
\end{gather}
Here we have used Lemma \ref{gamma(1,1)1vsiLemma} in the last line.

Next, we set $a=b^{-1}$ in Equation \ref{Complex15} to find
\begin{align*}
    \sum_ss\otimes\gamma(c,b)_s^a&=\sum_tt^a\otimes\gamma(c,1)_t\beta_2(b^{-1},c)\,,\\
\end{align*}
which implies that $\gamma(c,b)_s^b=(s^2)^{|b|}\gamma(c,1)_s\beta_2(b^{-1},c)$.
We replace $c$ with $a$ in the above equation, and then we use Equation \ref{GammaLeft} to find
\begin{gather*}
    \gamma(a,b)^b_s\;=\;(s^2)^{|b|}\gamma(a,1)_s\beta_2(b^{-1},a)\;=\;\overline{s}^{gab}\gamma(1,1)_1\beta_2(b^{-1},a)\,.
\end{gather*}


We replace the $\beta_2$ terms by using Equation \ref{Beta2} to get
\begin{align}
    \gamma(a,b)^b_s&=\frac{\overline{s}^{gab}\gamma(1,1)_1}{\alpha_2(b^{-1},a^{-1})}\,.\nonumber
\end{align}
By the multiplicativity, normality, and symmetry properties of $\alpha_2$, the above formula simplifies to the following final form:
\begin{align}
    \gamma(a,b)_s&=\frac{\overline{s}^{ga}\gamma(1,1)^b_1}{\alpha_2(a,b)^{gb}}\label{GammaFinal}\,.
\end{align}

Observe that Equation \ref{GammaFinal} contains the factor $\overline{s}^{ga}$.  This fact allows us to greatly reduce the complexity of Equations \ref{Complex12}-\ref{Complex15} by using a little algebra.
\begin{lemma}\label{Lem:ComplexReductionLemma}
      Let $u$ and $v$ represent words in the set $G\cup\{g\}$.  Consider the element
      \[P_{u,v}:=\tfrac12\left(1\otimes 1+i^u\otimes\overline{i}^v\right)\in\bb C\otimes_{\bb R}\bb C\,.\]
      This element is an idempotent and it satisfies that $P_{u,v}\cdot(\lambda\otimes1)=P_{u,v}\cdot(1\otimes\lambda^{uv})$. 
\end{lemma}
By identifying which projection $P_{u,v}$ appears on each side of Equations \ref{Complex12}-\ref{Complex15}, we use Lemma \ref{Lem:ComplexReductionLemma} to pass all the complex scalars across the tensor symbol via the conjugation $(-)^{uv}$.  In this way, we reduce Equations \ref{Complex12}-\ref{Complex15} to the following much simpler equations
\begin{align}
    \alpha_2(a,c)^{abc}\gamma(c,b)_1&=\gamma(c,a^{-1}b)_1\,,\label{12Reduced}\\
    \alpha_2(b,a)\gamma(c,b)_1^a&=\gamma(ca^{-1},b)_1\,,\label{13Reduced}\\
    \gamma(c,b)_1&=\beta_2(a,b)^{gc}\gamma(ac,b)_1\,,\label{14Reduced}\\
    \gamma(c,b)_1^a&=\gamma(c,ba)_1\beta_2(a,c)\,.\label{15Reduced}
\end{align}

Now, having this simplification, we set $b=1$ in Equation \ref{12Reduced} and continue to reduce
\begin{gather}
    \alpha_2(a,c)^{ac}\gamma(c,1)_1=\gamma(c,a^{-1})_1\nonumber\,,\\
    \alpha_2(a,c)^{ac}\gamma(1,1)_1=\frac{\gamma(1,1)_1^a}{\alpha_2(c,a^{-1})^{ga}}\nonumber\,,\\
    \alpha_2(c,a)^g\alpha_2(c,a^{-1})^{ga}=\frac{\gamma(1,1)_1^a}{\gamma(1,1)_1}\nonumber\,,\\
    \alpha_2(c,1)^{ga}=\frac{\gamma(1,1)_1^a}{\gamma(1,1)_1}\nonumber\,,\\
    1=\frac{\gamma(1,1)_1^a}{\gamma(1,1)_1}\,.\label{gamma111Real}
\end{gather}
Since Equation \ref{gamma111Real} holds for all $a\in G$, it follows that $\gamma(1,1)_1$ must be a real number.

We will come back to Equation \ref{gamma111Real} later in Equation \ref{OnlyTwoGammas}, but for now, let us examine Equation \ref{Complex16}. Setting $d=b=1$ and $a\neq1$, we use Lemma \ref{Lem:ComplexReductionLemma} to begin reducing Equation \ref{Complex16}
\begin{align}
    0&=\sum_{c,s,t}t\otimes s^{g}\cdot\gamma(c,1)^{g}_t\cdot\gamma(a,c)_s\nonumber\\
    0&=\sum_{c,s,t}t\otimes s^g\cdot\overline{t}^{c}\gamma(c,1)^{g}_1\cdot\overline{s}^{ga}\gamma(a,c)_1\nonumber\\
    0&=2\sum_{c,s}P_{1,c}\cdot\left(1\otimes\gamma(c,1)^{g}_1\cdot(s^2)^{|a|}\gamma(a,c)_1\right)\nonumber\\
    0&=\sum_{c,s}P_{1,c}\cdot\left(1\otimes\gamma(c,1)^{g}_1\cdot(s^2)^{|a|}\gamma(a,c)_1\right)\,.\label{NondegStep1}
\end{align}
Equation \ref{NondegStep1} is uninteresting when $|a|=1$, so suppose $|a|=0$ ($a\in A$).  With this assumption in hand, we continue to reduce Equation \ref{NondegStep1} 
\begin{align}
    0&=2\cdot\sum_{c}P_{1,c}\cdot\left(1\otimes\gamma(c,1)^{g}_1\gamma(a,c)_1\right)\nonumber\\
    0&=2\cdot\sum_{c}P_{1,c}\cdot\left(1\otimes\gamma(1,1)_1^{g}\left(\frac{\gamma(1,1)^c_1}{\alpha_2(a,c)^{gc}}\right)\right)\nonumber\\
    0&=2\cdot\sum_{c}P_{1,c}\cdot\left(1\otimes\frac{\gamma(1,1)^c_1}{\alpha_2(a,c)^{gc}}\right)\nonumber\\
    0&=\sum_{|c|=0}(1\otimes 1-i\otimes i)\cdot\left(1\otimes\frac{\gamma(1,1)^c_1}{\alpha_2(a,c)^{gc}}\right)\nonumber\\
    &\hspace{15mm}+\;\sum_{|c|=1}(1\otimes1+i\otimes i)\cdot\left(1\otimes\frac{\gamma(1,1)_1^c}{\alpha_2(a,c)^{gc}}\right)\nonumber\,.
\end{align}
For ease of notation, we let the tensor factor that depends on the index $c$ be denoted by $S(c)$, so that the above equation becomes
\begin{gather*}
    0=\sum_{|c|=0}(1\otimes 1-i\otimes i)\cdot\left(1\otimes S(c)\right)\;+\;\sum_{|c|=1}(1\otimes1+i\otimes i)\cdot\left(1\otimes S(c)\right)\,.
\end{gather*}
By taking the real and imaginary parts of the left tensor factor, we find that
\begin{gather}
    0=\sum_{|c|=0}S(c)\;+\;\sum_{|c|=1}S(c)\;,\hspace{2mm}\text{and}\hspace{5mm}
    0=-\sum_{|c|=0}iS(c)\;+\;\sum_{|c|=1}iS(c)\;.\nonumber
\end{gather}
This is only possible if both summations are zero.  Focusing on the $|c|=0$ sum, we continue reducing the equation
\begin{align}
    0&=\sum_{|c|=0}\frac{\gamma(1,1)^c_1}{\alpha_2(a,c)^{gc}}\nonumber\\
    0&=\sum_{|c|=0}\frac{\gamma(1,1)_1}{\alpha_2(a,c)^{g}}\nonumber\\
    0&=\sum_{|c|=0}\frac{1}{\alpha_2(a,c)}\nonumber\\
    0&=\sum_{|c|=0}\alpha_2(a,c)\,.\label{NondegStep3}
\end{align}
Since Equation \ref{NondegStep3} holds for any $a\neq 1\in A$, we conclude that $\alpha_2$ is nondegenerate when restricted to $A\times A$.  This implies in particular that $A$ must be abelian, but it provides even more.

\begin{lemma}\label{GeneralizedDihedral}
    The exact sequence $A\hookrightarrow G\twoheadrightarrow\bb Z/2\bb Z$ is necessarily split, and $\bb Z/2\bb Z$ acts on $A$ by inversion.
\end{lemma}

\begin{proof}
    Let $|w|=1$, and note that $w^2\in A$.  For all $a\in A$, Equation \ref{Multiplicativity} implies
    \[\alpha_2(a,w^2)\;=\;\alpha_2(a,w)\overline{\alpha_2(a,w)}\;=\;|\alpha_2(a,w)|^2\;=\;1\,.\]
    Thus by nondegeneracy of $\alpha_2$, $w^2=1\in G$.  We may use $1\mapsto w$ as the desired splitting.
    
    Conjugation by $w$ produces an automorphism $a\mapsto waw$ of $A$.  For any $b\in A$, we can use Equation \ref{Multiplicativity} again to find
    \begin{gather*}
        \alpha_2(b,w)\overline{\alpha_2(b,a)}\;=\;\alpha_2(b,wa)\;=\;\alpha_2(b,waww)\;=\;\alpha_2\big(b,waw\big)\alpha_2(b,w)\,,
    \end{gather*}
    which implies that $\alpha_2(b,a^{-1})\;=\;\alpha_2(b,a)^{-1}\;=\;\overline{\alpha_2(b,a)}\;=\;\alpha_2\big(b,waw\big)$.
    Finally, nondegeneracy implies that $waw=a^{-1}$.
\end{proof}

\begin{note}
    The groups $G$ described in Lemma \ref{GeneralizedDihedral} are known as generalized dihedral groups.
\end{note}

Earlier we observed that Equation \ref{gamma111Real} implies that $\gamma(1,1)_1\in\bb R$.  However not all real numbers will work because Equation \ref{Complex16} enforces a further restriction.  By setting $b=1$ and $d=a^{-1}$ in Equation \ref{Complex16}, we find that

\begin{align}
    1\otimes1&=\sum_{c,s,t}t\otimes s^{ga}\gamma(c,a^{-1})^{ga}_t\gamma(a,c)_s\nonumber\\
    &=\sum_{c,s}2P_{1,ac}\left(1\otimes\gamma(c,a^{-1})^{ga}_1\gamma(a,c)_1\right)\nonumber\\
    &=2\sum_{c}2P_{1,ac}\left(1\otimes\gamma(c,a^{-1})^{ga}_1\gamma(a,c)_1\right)\,.\nonumber
\end{align}

By taking the real part of the left tensor factor, this becomes
    
\begin{align*}
    1&=2\sum_{c}\gamma(c,a^{-1})^{ga}_1\gamma(a,c)_1\\
    &=2\sum_{c}\left(\frac{\gamma(1,1)^{a}_1}{\alpha_2(c,a^{-1})^{ga}}\right)^{ga}\left(\frac{\gamma(1,1)^c_1}{\alpha_2(a,c)^{gc}}\right)\\
    &=2\gamma(1,1)^2_1\sum_{c}\frac{1}{\alpha_2(c,1)}\\
    &=2\gamma(1,1)^2_1|G|\,.
\end{align*}
By solving for $\gamma(1,1)_1$, we find that
\begin{align}
    \gamma(1,1)_1&=\pm\frac{1}{\sqrt{2|G|}}\,.\label{OnlyTwoGammas}
\end{align}
This requirement for $\gamma(1,1)_1$ is analogous to the previous requirements that appear in the classification for the non-split real and quaternionic Tambara-Yamagami categories.

By combining Equations \ref{Multiplicativity} and \ref{Symmetry}, it follows that

\begin{gather}
    \alpha_2(a,b^{-1})\;=\;\frac{1}{\alpha_2(a,b)^b}\;\;\text{and}\;\;\alpha_2(a^{-1},b)\;=\;\frac{1}{\alpha_2(a,b)^a}\,.\label{Inversion4Alpha_2}
\end{gather}
We can combine the equations above with Equation \ref{Beta2} to get
\begin{gather}
    \beta_2(a,b)\;=\;\alpha_2(a,b)^b\,.\label{Beta2Final}
\end{gather}

\subsection{Classification}

With the above reductions complete, we can give necessary conditions for the pentagon equations to have solutions in this case.

Let $G\cong A\rtimes\bb Z/2\bb Z$ be a finite generalized dihedral group. Let $\mathbb C^\times_*$ denote the complex units as a $G$-module with the canonical action of $G$ that factors through $\bb Z/2\bb Z$, where $\bb Z/2\bb Z$ acts on $\bb C^\times$ by complex conjugation.

\begin{definition}
    A \emph{bicocycle} for $G$ is a function $\chi:G\times G\to \mathbb C^\times_*$ that satisfies the following properties
    \begin{gather*}
        \chi(a, bc)=\chi(a,b)\chi(a,c)^b\,,\hspace{5mm}\text{ and }\hspace{5mm}\chi(ab,c)=\chi(a,c)^b\chi(b,c)\,.
    \end{gather*}
    A bicocycle $\chi$ is said to be \emph{symmetric} with respect to $(-)^g\in\text{Gal}(\mathbb C/\mathbb R)$ if it satisfies the additional relation
    \begin{align*}
        \chi(a,b)&=\chi(b,a)^{gab}\,.
    \end{align*}
\end{definition}

\begin{theorem}\label{Thm:TYComplex}
        Let $\tau=\sfrac{\pm1}{\sqrt{2|G|}}$, let $(-)^g\in\text{Gal}(\mathbb C/\mathbb R)$, and let $\chi:G\times G\to \mathbb C^\times_*$ be a symmetric bicocycle on $G$ with respect to $(-)^g$, whose restriction $\chi\mid_{A\times A}$ is a nongedegerate bicharacter.
    
    
    A quadruple of such data gives rise to a non-split Tambara-Yamagami category $\s C_{\bb C}(G,g,\chi,\tau)$, with $\End(\1)\cong\bb R$ and $\End(m)\cong\bb C$.  Furthermore, all equivalence classes of such categories arise in this way. More explicitly, two categories \newline $\s C_{\bb C}(G,g,\chi,\tau)$ and $\s C_{\bb C}(G',g',\chi',\tau')$ are equivalent if and only if $g=g'$, and there exists 
    the following data:
    \begin{enumerate}[label = \roman*, align=CenterWithParen, labelwidth=1.5em]
        \item an isomorphism $f:G\to G'$,
        \item a map $(-)^h:\bb C\to\bb C$, either the identity or complex conjugation,
        \item a scalar $\lambda\in S^1\subset \mathbb C$,
    \end{enumerate}
    satisfying the following conditions for all $a,b\in G$
    \begin{gather}
        \chi'\Big(f(a),f(b)\Big)=\frac{\lambda\cdot\lambda^{ab}}{\lambda^a\cdot\lambda^b}\cdot\chi(a,b)^h\;,\label{EquivCond1}\\
        \frac{\tau'}{\tau}=\frac{\lambda}{\lambda^g}\label{EquivCond2}\,.
    \end{gather}
\end{theorem}

\begin{proof}
    To prove the existence, we will construct the category $\s C=\s C_{\bb C}(G, g, \chi,\tau)$, and show that its associator satisfies the pentagon Equations \ref{Complex1}-\ref{Complex16}.
    
    For a given $a\in G$, let us denote the two projections $m\otimes m\to a$ as 
    $\pi_{a,s}=[a](\id_m\otimes \overline{s})$, and the two inclusion as $\iota_{a,t}=(\id_m\otimes t)[a]^\dagger$.
    The components of the associator for $\s C$ are defined by the following equations
    \begin{gather*}
	    \alpha_{a,b,c}=\id_{abc},\\
	    \alpha_{a,b,m}=\alpha_{m,a,b}\;=\;\id_m,\\
	    \alpha_{a,m,b}=\chi(a,b)^{ab}\cdot\id_m,\\
	    \alpha_{m,m,a}=\alpha_{a,m,m}\;=\;\id_{m\otimes m},\\
	    \alpha_{m,a,m}=\sum_{\substack{b\in G\\t\in\{1,i\}}}\left(\id_m\otimes \chi(a,b)^bt^b\right)\big(\iota_{b,1}\pi_{b,t}\big),\\
		\alpha_{m,m,m}=\sum_{\substack{a,b\in G\\s,t\in\{1,i\}}}(\id_m\otimes\iota_{a,t})\circ\left(\frac{\overline{s}^{gab}t^{b}\tau}{\chi(a,b)^g}\right)\circ(\pi_{b,s}\otimes\id_m)\,.
	\end{gather*}
    The left and right unitors $\ell_X$ and $r_X$ are identities for all simple objects $X$.
    
    The formulas above are designed so that the corresponding $\alpha$, $\alpha_i$'s, $\beta_j$'s and $\gamma$ of the category $\s C$ satisfy the following identities
    \begin{gather}
        \alpha\equiv\alpha_1\equiv\alpha_3\equiv\beta_1\equiv\beta_3\equiv 1\label{easycoeff}\,,\\
        \alpha_2(a,b)=\chi(a,b)=\beta_2(a,b)^b,\label{ab2IsChi}\\
        \gamma(a,b)_s=\frac{\overline{s}^{ga}\tau}{\chi(a,b)^{gb}}\,.\label{gammaHastheSga}
    \end{gather}
    We have already shown that the above relations are necessary for Equations \ref{Complex1}-\ref{Complex16} to be valid, and so we proceed to show that they are sufficient as well.
    
    Equations \ref{Complex1}, \ref{Complex2}, \ref{Complex3}, \ref{Complex6}, and \ref{Complex7} follow immediately from Equations \ref{easycoeff}.  Equations \ref{Complex4}, \ref{Complex5}, and \ref{Complex8} - \ref{Complex11} follow from the fact that $\chi$ is a symmetric bicocycle.
    
    Equations \ref{Complex12} - \ref{Complex14} are all similar to one another.  Since the functions $\alpha, \alpha_1, \alpha_3, \beta_1,$ and $\beta_3$ are trivial by Equations \ref{easycoeff}, and since $\gamma(a,b)_s$ has the factor $\overline{s}^{ga}$ as in Equation \ref{gammaHastheSga}, we may equivalently verify Equations \ref{12Reduced} - \ref{15Reduced} instead.  These reduced equations are immediate from Equations \ref{ab2IsChi} and \ref{gammaHastheSga}.
    
    Finally it is necessary to prove Equation \ref{Complex16}.  We begin by reducing the right-hand side.
    \begin{align}
        &\sum_{c,s,t}t\otimes s^{gbd}\cdot\beta_2(c,b)\cdot\gamma(c,d)^{gbd}_t\cdot\gamma(a,c)_s\nonumber\\
        &=\tau^2\sum_{c,s,t}t\otimes s^{gbd}\overline{s}^{ga}\overline{t}^{cbd}\cdot\frac{1}{\chi(c,b^{-1})\chi(c,d)^{b}\chi(c,a)^{a}}\nonumber\\
        &=\tau^2\sum_{c,t}t\otimes \sum_{s}\left(s^{gbd}\overline{s}^{ga}\right)\overline{t}^{cbd}\cdot\frac{1}{\chi(c,b^{-1}d)\chi(c,a)^{a}}\nonumber\\
        &=\tau^2\left(1+(-1)^{|abd|}\right)\sum_{c,t}t\otimes \overline{t}^{cbd}\cdot\frac{1}{\chi(c,ab^{-1}d)^a}\,.\label{MidwayProvingComplex16}
    \end{align}
    If $|abd|=1$, the entire expression is zero, and this matches the left-hand side of Equation \ref{Complex16}, so there is nothing to prove. When $d=ba^{-1}$, Equation \ref{MidwayProvingComplex16} becomes
    \begin{gather*}
        2\tau^2\sum_{c,t}t\otimes \overline{t}^{ac}
        \;=\;2\tau^2\sum_{c}\left(1\otimes1\;+\;i\otimes\overline{i}^{ac}\right)
        \;=\;2\tau^2|G|\left(1\otimes1\right)
        \;=\;1\otimes 1\,.
    \end{gather*}
    This also agrees with the left-hand side of Equation \ref{Complex16}.  Thus the only case left to analyze is the case when both $|abd|=0$ and $d\neq ba^{-1}$.  Let us set $f=ab^{-1}d\in A$ in Equation \ref{MidwayProvingComplex16} to continue the computation in this final case
    \begin{align}
        (\ref{MidwayProvingComplex16})\hspace{5mm}&=\tau^2\left(2\right)\sum_{c,t}t\otimes \overline{t}^{ca}\cdot\frac{1}{\chi(c,f)^a}\nonumber\\
        &=2\tau^2\sum_{c}2P_{1,ca}\left(1\otimes\frac{1}{\chi(c,f)^a}\right)\nonumber\\
        &=2\tau^2\sum_{c}2P_{1,ca}\left(\frac{1}{\chi(c,f)^c}\otimes1\right)\nonumber\\
        &=2\tau^2\sum_{c}2P_{1,ca}\big(\chi(c^{-1},f)\otimes1\big)\nonumber\\
        &=2\tau^2\sum_{c}2P_{1,ca}\big(\chi(c,f)\otimes1\big)\,.\label{AlmostDoneComeplex16}
    \end{align}
    Since $G$ is generalized dihedral, we can write every $c\in G$ as $c=uw^\epsilon$, where $u\in A$, $|w|=1$ and $\epsilon\in\{0,1\}$.  Using this description, Expression \ref{AlmostDoneComeplex16} becomes
    \begin{align*}
        &2\tau^2\left(2P_{1,a}\sum_{u\in A}\chi(u,f)\otimes1\;+\;2P_{1,wa}\sum_{u\in A}\chi(uw,f)\otimes1\right)\\
        &=2\tau^2\left(2P_{1,wa}\sum_{u\in A}\chi(uw,f)\otimes1\right)\\
        &=2\tau^2\left(2P_{1,wa}\sum_{u\in A}\overline{\chi(u,f)}\chi(w,f)\otimes1\right)\;=\;0.
    \end{align*}
    Here we have used nondegeneracy of $\chi\mid_{A\times A}$ to eliminate the two summations.  This again matches the left-hand side of Equation \ref{Complex16}, so this equation is satisfied in all cases.  This completes the proof of the pentagon equations, and thus establishes the existence of the monoidal categories $\s C_{\bb C}(G,g,\chi,\tau)$.
    
    \vspace{5mm}
    
    Now suppose there is an $\bb R$-linear monoidal equivalence $$(F,J):\s C:=\s C_{\bb C}(G,g,\chi,\tau)\to\s C_{\bb C}(G',g',\chi',\tau')=:\s C'\,.$$  Since $F$ is an equivalence, it must send $m\in\s C$ to $m'\in\s C'$.  Since $(F,J)$ is monoidal, it must restrict to a group isomorphism $f:G\to G'$.  Because of this, we may assume that $G=G'$, and that $f:G\to G$ is an automorphism.
    
    Since $F$ is $\bb R$-linear, it must induce an $\bb R$-linear isomorphism $\bb C\cong\End(m)\to\End(m')\cong\bb C$.  Since the Galois group of $\mathbb C$ over $\mathbb R$ consists of the identity and complex conjugation, let us denote $\lambda^h:=F(\lambda)\in\bb C\cong\End(m')$.  Since monoidal functors preserve duals, we get $g=g'$. In particular, the property of $m$ being \emph{directly self-dual}, that is, $g=\id$, or \emph{conjugately self-dual}, that is, $g=\overline{(\;)}$, is an invariant of the monoidal equivalence class of $\s C_{\bb C}(G,g,\chi,\tau)$.
    
    The monoidal structure map $J$ is required to satisfy a hexagon relation.  There is one hexagon relation for every sequence of three simple objects.  Since the simple objects can either be invertible or $m$, there are $2^3$ equations that must be satisfied.  These requirements are analogous to Equations \ref{Basis1} - \ref{Basis7} together with one additional equation relating to $\gamma$.  In our current context, these $8$ equations are as follows
    \begin{align}
        1&=\delta\theta\,,\label{Equivalence1}\\
        \theta&=\delta^L(\psi)\,,\label{Equivalence2}\\
        \chi'\Big(f(a),f(b)\Big)&=\frac{\psi(b)\varphi(a)^b}{\psi(b)^a\varphi(a)}\cdot\chi(a,b)^h\label{Equivalence3}\,,\\
        \theta&=\delta^R(\varphi)\label{Equivalence4}\,,\\
        \varphi(a)^{gab}\omega(b)&=\omega(a^{-1}b)\theta(a,a^{-1}b)\label{Equivalence5}\,,\\
        \chi'\Big(f(a),f(b)\Big)^b&=\frac{\omega(b)^a\varphi(a)}{\omega(b)\psi(a)^{gab}}\cdot\chi(a,b)^{hb}\label{Equivalence6}\,,\\
        \theta(ba^{-1},a)\omega(ba^{-1})&=\psi(a)\omega(b)^a\label{Equivalence7}\,,\\
        \sum_t\frac{t}{\psi(a)^{gb}}\otimes\frac{\gamma'\Big(f(a),f(b)\Big)_{t}}{\omega(a)}&=\sum_s\frac{s^h}{\omega(b)}\otimes\frac{\gamma(a,b)^h_s}{\varphi(b)}\label{Equivalence8}\,.
    \end{align}
    By using Equation \ref{gammaHastheSga}, we can once again apply Lemma \ref{Lem:ComplexReductionLemma} to reduce Equation \ref{Equivalence8} to the following
    \begin{align}
        \frac{\tau'}{\psi(a)^{ab}\omega(a)\chi'\Big(f(a),f(b)\Big)^{gb}}&=\frac{\tau}{\omega(b)^{ga}\varphi(b)\chi(a,b)^{ghb}}\label{Equivalence8Reduced}\,.
    \end{align}

By setting $a=1$ in Equations \ref{Equivalence5}, \ref{Equivalence6}, and \ref{Equivalence7}, we get 
\begin{align}
    \varphi(1)&=\theta(1,b),\label{RealphipsiObs1}\\
    \psi(1)&=\theta(b,1),\;\;
    \label{RealphipsiObs2}\\
    \varphi(1)&=\psi(1)\;.\label{RealphipsiObs3}
\end{align}

In exactly the same way we were able to normalize Equations \ref{Complex1} - \ref{Complex16} using a change of basis, we can simplify Equations \ref{Equivalence1} - \ref{Equivalence8} by using a monoidal natural isomorphism $\mu:(F,J)\to(F,J')$.
By monoidality, the components of $\mu$ must satisfy the following equations
\begin{align}
    \theta'&=\theta\cdot\delta(\mu_{-})\,,\label{Isomorphism1}\\
    \varphi'(a)&=\frac{\mu_m\mu_a}{\mu_m^a}\cdot\varphi(a)\,,\label{Isomorphism2}\\
    \psi'(a)&=\frac{\mu_m\mu_a}{\mu_m^a}\cdot\psi(a)\,,\label{Isomorphism3}\\
    \omega'(a)&=\frac{\mu_m^{ga}\mu_m}{\mu_a}\cdot\omega(a)\,,\label{Isomorphism4}
\end{align}
where $(\theta',\psi',\varphi',\omega')$ are the coefficients of our new tensorator $J'$.

Equations \ref{RealphipsiObs1} - \ref{RealphipsiObs3} imply that $\psi(1)=\varphi(1)$ is real. By choosing $\mu$ such that $\mu_1=\varphi(1)^{-1}$, and all other $\mu_X$ are trivial, we can assume without loss of generality that $\psi(1)=\varphi(1)=\theta(b,1)=\theta(1,b)=1$.

By taking the norm of both sides of Equation \ref{Equivalence2}, we find that $|\theta|=\delta|\psi|$.  If we set $\mu_a:=|\psi(a)|^{-1}$, this allows us to assume that $|\theta|=1$, so $\theta(a,b)\in\{\pm1\}$.

Now we proceed by setting $b=a$ in Equations \ref{Equivalence5} and \ref{Equivalence7} and we get
\begin{align}
    \varphi(a)&=\frac{\omega(1)^g}{\omega(a)^g}\label{VarpInTermsOfOmega}\,,\\
    \psi(a)&=\frac{\omega(1)}{\omega(a)^a}\label{PsiInTermsOfOmega}\,.
\end{align}

Equation \ref{VarpInTermsOfOmega} allows us to express $\omega$ in terms of $\varphi$. Using this, we expand Equation \ref{Equivalence6} for arbitrary $a$ and $b$ to find that

\begin{gather}
    \frac{\varphi(a)}{\varphi(a)^b}\;=\;\frac{\varphi(b)}{\varphi(b)^a}\;.\label{VarpIsTrivialStep1}
\end{gather}

Equation \ref{VarpIsTrivialStep1} implies that there is some $\lambda\in \mathbb C^\times_*$ such that $\varphi^2=\delta\lambda^2$.  By taking square roots, we find that $\varphi(a)=\pm(\delta\lambda)(a)$ with the sign possibly depending on $a$.  We have used the magnitude of $\mu_a$, but we are still free to use the sign of $\mu_a$ in Equation \ref{Isomorphism2} to ensure that
\begin{gather}
    \varphi(a)=\frac{\lambda}{\lambda^a}\,,\label{VarpIsTrivialStep2}
\end{gather}
which eliminates the sign ambiguity.  Next we can use $\mu_m=\lambda^{-1}$ in Equation \ref{Isomorphism2} to assume that $\varphi\equiv1$.

In light of $\varphi$ being trivial, Equation \ref{VarpInTermsOfOmega} implies that $\omega(a)=\omega(1)$, so $\omega$ is constant.  This combines with Equation \ref{PsiInTermsOfOmega} to yield a nice formula for $\psi$ in terms of the constant $\omega(1)$.  In summary, after normalization we arrive at the following formulas
\begin{align}
    \varphi(a)&=1\,,\label{VarpIsTrivialStep3}\\
    \omega(a)&=\omega(1)\,,\label{OmegaIsConstant}\\
    \psi(a)&=\frac{\omega(1)}{\omega(1)^a}\label{PsiIsACoboundary}\,.
\end{align}

These new formulas can be combined with Equation \ref{Equivalence3} to produce
\begin{gather}
    \chi'\Big(f(a),f(b)\Big)=\frac{\omega(1)\cdot\omega(1)^{ab}}{\omega(1)^a\cdot\omega(1)^b}\cdot\chi(a,b)^h\;,\label{EquivCond1Rediscovered}
\end{gather}

They can also be combined with Equations \ref{Equivalence3} and \ref{Equivalence6} to imply that
\begin{gather}
    \omega(1)^2\;=\;\big(\omega(1)^2\big)^g\,.\label{Omega1SquaredIsgInvariant}
\end{gather}
Finally Equations \ref{VarpIsTrivialStep3}, \ref{OmegaIsConstant}, \ref{PsiIsACoboundary}, and \ref{Omega1SquaredIsgInvariant} can be used to reduce Equation \ref{Equivalence8Reduced} to derive
\begin{gather}
    \frac{\tau'}{\tau}\;=\;\frac{\omega(1)}{\omega(1)^g}\;.\label{EquivCond2Rediscovered}
\end{gather}

By setting $\omega(1)=\lambda$, the reader will recognize Equations \ref{EquivCond1Rediscovered} and \ref{EquivCond2Rediscovered} as Conditions \ref{EquivCond1} and \ref{EquivCond2} respectively from the statement of the theorem.  Thus we have shown that a generic equivalence forces $g=g'$, gives rise to the data $(f,h,\lambda)$ stated in the theorem, and makes Conditions \ref{EquivCond1} and \ref{EquivCond2} necessary.

We now turn to the question of sufficiency.  Suppose that $G=G'$, $g=g'$, and that the data $(f,h,\lambda)$ are given.  Then the pair $(f,h)$ uniquely determines the underlying functor $F:\mathcal C_{\mathbb C}(G,g,\chi,\tau)\to\mathcal C_{\mathbb C}(G,g,\chi',\tau')$.  Define the functions
\begin{gather*}
    J_{a,b}=\id_{f(a)\otimes f(b)}\;,\hspace{4mm}
    J_{a,m}=\id_{f(a)\otimes m}\;,\hspace{4mm}
    J_{m,a}=\left(\frac{\lambda}{\lambda^a}\right)\otimes\id_{f(a)}\;,\hspace{4mm}
    J_{m,m}=\id_{m}\otimes\lambda\;.
\end{gather*}
These are the components of a monoidal structure map, and are clearly isomorphisms since $\lambda\neq0$.  This monoidal structure map $J$ can be described in terms of complex-valued coordinate functions $(\theta,\varphi,\psi,\omega)$ as follows
\begin{gather*}
    \theta(a,b)=1\;,\hspace{6mm}
    \varphi(a)=1\;,\hspace{6mm}
    \psi(a)=\frac{\lambda}{\lambda^a}\;,\hspace{6mm}
    \omega(a)=\lambda\;.
\end{gather*}
The fact that these coefficient functions satisfy Equations \ref{Equivalence1} through \ref{Equivalence8} is easy to check.  We comment that Condition \ref{EquivCond1} is used to prove Equation \ref{Equivalence3}, Condition \ref{EquivCond2} is used to prove Equations \ref{Equivalence5} and \ref{Equivalence7}, and both Conditions \ref{EquivCond1} and \ref{EquivCond2} are necessary to prove Equations \ref{Equivalence6} and \ref{Equivalence8}.  Thus the pair $(F,J)$ is a monoidal equivalence, and the theorem is proven.

\end{proof}

\begin{example}
    The simplest dihedral group is the group $G=D_{2\cdot1}\cong\mathbb Z/2\mathbb Z$.  This corresponds to the case where $A$ is the trivial group.  Let us denote the nontrivial element of $G$ by $w$.  By using Theorem \ref{Thm:TYComplex}, we find that there are exactly four categories with complex $m$ and group $G$.  Upon base extension to $\mathbb C$, these categories become pointed, with fusion rules corresponding to $\mathbb Z/4\mathbb Z$ or $(\mathbb Z/2\mathbb Z)^2$.
    
    When $g$ conjugates, the symmetry of $\chi$ implies that $\chi(w,w)=\pm1$.  Condition \ref{EquivCond2} implies that it is possible for $\tau'=-\tau$ by setting $\lambda=\pm i$.  This shows that there is an equivalence $\mathcal C(G,g,\chi,\tau)\simeq\mathcal C(G,g,\chi,-\tau)$.  Thus the only relevant invariant of these categories is the number $\chi(w,w)$.  The case where $\chi(w,w)=1$ becomes $\mathbb C\text{-}\Vec_{\mathbb Z/4\mathbb Z}$ upon extension to $\mathbb C$.  The case where $\chi(w,w)=-1$ becomes $\mathbb C\text{-}\Vec_{\mathbb Z/4\mathbb Z}^{\upsilon^2}$, where $\upsilon^2(a^i,a^j,a^k)=(-1)^{i\cdot\lfloor\frac{j+k}{4}\rfloor}$ represents the unique cohomology class of order two in $H^3\big(\mathbb Z/4\mathbb Z;\mathbb C^\times\big)$.  Using the descent theory of \cite{etingofDescentAndForms}, it can be verified that these are the only two cohomology classes in  for which the corresponding pointed category admits a real form.
    
    When $g$ doesn't conjugate, Condition \ref{EquivCond2} implies that $\tau'=\tau$.  Theorem \ref{Thm:TYComplex} implies that the scalar $\chi(w,w)$ can always be normalized to be $1$ by choosing any $\lambda$ such that $\lambda^4=\chi(w,w)^{-1}$.  Thus the only relevant invariant of these categories is $\tau=\pm\tfrac12$.  The case where $\tau=\tfrac12$ becomes $\mathbb C\text{-}\Vec_{(\mathbb Z/2\mathbb Z)^2}$ upon extension to $\mathbb C$, while the case where $\tau=-\tfrac12$ becomes $\mathbb C\text{-}\Vec_{(\mathbb Z/2\mathbb Z)^2}^{\xi}$ upon extension to $\mathbb C$, where $\xi(a^ib^j,a^kb^\ell,a^mb^n)=(-1)^{ikm+j\ell n}$ is a representative cocycle in $H^3((\mathbb Z/2\mathbb Z)^2;\mathbb C^\times\big)$ corresponding to the associator.  Descent theory again verifies that these are the only two associators that allow the category to have a real form.
\end{example}

\begin{example}
    When $A=\mathbb Z/n\mathbb Z$, the group is $G=D_{2\cdot n}$, the dihedral group of order $2n$.  If $a\in A$ is a generator, then $\chi(a,a)$ must be a primitive $n^\text{th}$ root of unity.  The symmetry condition then implies that $g$ is allowed to conjugate only in the case where $A=\mathbb Z/2\mathbb Z$, and for all other cases $g$ must not conjugate.
    
    Note that this restriction is only for classical dihedral groups.  For generalized dihedral groups where $A$ is non-cyclic, there are typically more options.
\end{example}

\begin{proposition}\label{Prop:ComplexCatsAreRigid}
    The categories $\mathcal C_{\mathbb C}(G,g,\chi,\tau)$ are rigid.
\end{proposition}

\begin{proof}
    Similarly to Proposition \ref{Prop:QuatCatsAreRigid}, we may take $\text{ev}_m=[1]$ and $\coev_{m}=\tau^{-1}[1]^\dagger$, and the duality equations follow from a direct computation.  As before, the only missing ingredient is the following formula for the inverse of the associator
    \[\alpha_{m,m,m}^{-1}\;=\;\sum_{\substack{a,b\in G\\s,t\in\{1,i\}}}(\iota_{a,t}\otimes1)\circ\left(\overline{s}^{a}t^{gab}\tau\chi(b,a)^g\right)\circ(1\otimes\pi_{b,s})\,.\]
    
\end{proof}

\section{Analysis of the Complex Galois Case\label{Sec:ComplexGalois}}



In this section, we will construct the non-split Tambara-Yamagami categories \newline $\mathcal C_{\overline{\mathbb C}}(A,\chi)$, where all simple objects are complex, and $m$ is the unique Galois nontrivial simple object.  The analysis follows the same pattern as before, but now we keep track of each time the Galois action of $m$ is applied.  The pentagon equations are as follows

\begin{align}
\delta\alpha&=1\,,\label{FullComplex1}\\
\delta\alpha_3&=\alpha^{-1}\,,\label{FullComplex2}\\
\delta\alpha_1&=\overline{\alpha}\,,\label{FullComplex3}\\
\alpha_2(a,bc)&=\alpha_2(a,c)\alpha_2(a,b)\,,\label{FullComplex4}\\
\alpha_2(ab,c)&=\alpha_2(b,c)\alpha_2(a,c)\,,\label{FullComplex5}\\
\overline{\alpha(a,b,b^{-1}a^{-1}c)}\beta_1(ab,c)&=\beta_1(b,a^{-1}c)\beta_1(a,c)\overline{\alpha_3(a,b)}\,,\label{FullComplex6}\\
\beta_3(ab,c)\overline{\alpha(cb^{-1}a^{-1},a,b)}&=\alpha_1(a,b)\beta_3(b,c)\beta_3(a,cb^{-1})\,,\label{FullComplex7}\\
\beta_2(b,c)&=\beta_2(b,a^{-1}c)\overline{\alpha_2(a,b)}\,,\label{FullComplex8}
\end{align}

\begin{align}  
\beta_2(a,c)&=\alpha_2(a,b)\beta_2(a,cb^{-1})\,,\label{FullComplex9}\\
\beta_1(a,c)\beta_3(b,c)&=\beta_3(b,a^{-1}c)\overline{\alpha(a,a^{-1}cb^{-1},b)}\beta_1(a,cb^{-1})\,,\label{FullComplex10}\\
\beta_2(a,c)\beta_2(b,c)&=\alpha_3(a,b)\beta_2(ab,c)\overline{\alpha_1(a,b)}\,,\label{FullComplex11}\\
\alpha_2(a,c)\gamma(c,b) &= \overline{\beta_1(a,b)}\alpha_3(a,a^{-1}b)\gamma(c,a^{-1}b)\,,\label{FullComplex12}\\
\alpha_2(b,a)\gamma(c,b) &= \beta_3(a,c)\alpha_1(ca^{-1},a)\gamma(ca^{-1},b)\,,\label{FullComplex13}\\
\alpha_1(a,c)\gamma(c,b) &= \overline{\beta_2(a,b)}\beta_1(a,ac)\,,\gamma(ac,b)\,,\label{FullComplex14}\\
\alpha_3(b,a)\gamma(c,b) &= \beta_2(a,c)\overline{\beta_3(a,ba)}\gamma(c,ba)\,,\label{FullComplex15}\\
\delta_{d,ba^{-1}}\beta_3(a,b)\beta_1(ba^{-1},b)&=\sum_{c}\beta_2(c,b)\overline{\gamma(c,d)}\gamma(a,c)\,.\label{FullComplex16}
\end{align}

An equivalence $(F,J):\mathcal C\to\mathcal C'$, amounts to having an isomorphism $f:A\to A'$, an automorphism $(-)^h\in\text{Gal}(\bb C/\bb R)$, and a collection $\{\theta,\varphi,\psi,\omega\}$ of complex valued functions that act as the coefficients of the tensorator $J_{X,Y}:F(X)\otimes F(Y)\to F(X\otimes Y)$.
\begin{align}
    f^*\alpha'&=\alpha^h\cdot\delta\theta\,,\label{FullCBasis1}\\
    \alpha_1'\big(f(a),f(b)\big)&=\alpha_1(a,b)^h\cdot\frac{\psi(ab)\overline{\theta}(a,b)}{\psi(a)\psi(b)}\,,\label{FullCBasis2}\\
    \alpha_2'\big(f(a),f(b)\big)&=\alpha_2(a,b)^h\,,\label{FullCBasis3}\\
    \alpha_3'\big(f(a),f(b)\big)&=\alpha_3(a,b)^h\cdot\frac{\varphi(b)\varphi(a)}{\varphi(ab)\theta(a,b)}\,,\label{FullCBasis4}\\  
    \beta_1'\big(f(a),f(b)\big)&=\beta_1(a,b)^h\cdot\frac{\omega(a^{-1}b)\theta(a,a^{-1}b)}{\varphi(a)\omega(b)}\,,\label{FullCBasis5}\\
    \beta_2'\big(f(a),f(b)\big)&=\beta_2(a,b)^h\cdot\frac{\overline{\varphi}(a)}{\psi(a)}\,,\label{FullCBasis6}\\
    \beta_3'\big(f(a),f(b)\big)&=\beta_3(a,b)^h\cdot\frac{\overline{\psi}(a)\omega(b)}{\theta(ba^{-1},a)\omega(ba^{-1})}\,,\label{FullCBasis7}\\
    \gamma'\big(f(a),f(b)\big)&=\gamma(a,b)^h\cdot\frac{\psi(a)\overline{\omega}(a)}{\varphi(b)\omega(b)}\,.\label{FullCBasis8}
\end{align}

The sequence of deductions that follow is only a superficial modification of the original argument of \cite{TAMBARA1998692}, but we include it here for completeness.  Begin by using the identity functor for $F$, so that $f$ and $h$ are trivial, and Equations \ref{FullCBasis1}-\ref{FullCBasis8} reduce to change of basis formulas.

By setting $\theta=\alpha_3$ and $\varphi\equiv1$ in Equation \ref{FullCBasis4}, we may assume that $\alpha\equiv1$ and $\alpha_3\equiv1$.  Setting $\overline{\psi}(a)=\beta_2(a,1)$ in Equation \ref{FullCBasis6} allows us to assume that $\beta_2(a,1)\equiv1$.

Next, we substitute $b=1$ into Equation \ref{FullCBasis5} to find that
\[\beta_1'(a,1)=\frac{\omega(a^{-1})\theta(a,a^{-1})}{\omega(1)}\cdot\beta_1(a,1)\,.\]
By rearranging this formula, we find that if we set
\[\omega(a^{-1}):=\frac{\omega(1)}{\theta(a,a^{-1})\beta_1(a,1)}\,,\]
this allows us to assume that $\beta_1(a,1)\equiv1$.  Upon substituting $c=1$ into Equation \ref{FullComplex6}, this new normalization shows that $\beta_1\equiv1$.

Setting $a=c$ in Equation \ref{FullComplex8}, and $b=c$ in Equation \ref{FullComplex9} imply that
\[\alpha_2(b,a)\;=\;\beta_2(b,a)\;=\;\overline{\alpha_2(a,b)}\,.\]
By Equations \ref{FullComplex4} and \ref{FullComplex5}, $\beta_2$ is a bicharacter, and thus Equation \ref{FullComplex11} implies that $\alpha_1\equiv1$.

Equation \ref{FullComplex10} implies that $\beta_3(a,b)=\beta_3(a,1)$, and by setting $a=b$ in Equation \ref{FullComplex12}, we get that $\gamma(a,b)=\alpha_2(a,b)\gamma(a,1)$.
We can combine these with Equations \ref{FullComplex13} and \ref{FullComplex14} to obtain
\[\gamma(c,1)\;=\;\beta_3(a,1)\gamma(ca^{-1},1)\,, \text{ and } \gamma(c,1)\;=\;\gamma(ac,1)\,.\]
This shows that $\beta_3\equiv1$ and $\gamma(a,b)\;=\;\alpha_2(a,b)\gamma(1,1)$.

With all these observations in place, we reduce Equation \ref{FullComplex16} to produce
\[\delta_{d,ba^{-1}}\;=\;\gamma(1,1)^2\sum_{c\in A}\alpha_2\big(c,ba^{-1}d^{-1}\big)\,,\]
which is equivalent to $\alpha_2$ being nondegenerate, and $\gamma(1,1)^2|A|=1$.


\begin{theorem}\label{Thm:TYFullComplex}
    Let $A$ be a finite group, and let $\chi:A\times A\to \mathbb C^\times$ be a nondegenerate skew-symmetric bicharacter.  Such a pair $(A,\chi)$ gives rise to a non-split Tambara-Yamagami category $\s C_{\overline{\bb C}}(A,\chi)$, with $\End(X)\cong\bb C$ for every simple object $X$.  Furthermore, all equivalence classes of such categories arise in this way.  Two categories $\s C_{\overline{\bb C}}(A,\chi)$ and $\s C_{\overline{\bb C}}(A',\chi')$ are equivalent if and only there exist isomorphisms:
    \begin{enumerate}[label = \roman*, align=CenterWithParen, labelwidth=1.5em]
        \item an isomorphism $f:A\to A'$, and
        \item $(-)^h:\bb C\to\bb C$ (either the identity or complex conjugation),
    \end{enumerate}
    such that $\chi'\big(f(a),f(b)\big)=\chi(a,b)^h$ for all $a,b\in A$.
\end{theorem}

\begin{proof}
    We begin by letting $\tau=\pm\tfrac{1}{\sqrt{A}}$, then defining an auxiliary category $\mathcal C_{\overline{\mathbb C}}(A,\chi,\tau)$ by giving it the desired fusion rules, asserting that all simple objects have $\End(X)\cong\mathbb C$, and requiring $m$ to be Galois nontrivial.  We define the associators by the following equations
    \begin{gather*}
        \alpha_{a,b,c}=\id_{abc}\,,\\
	    \alpha_{a,b,m}=\alpha_{m,b,c}=\id_{m}\,,\\
	    \alpha_{a,m,c}=\chi(a,c)\cdot\id_{m},\\
	    \alpha_{a,m,m}=\alpha_{m,m,c}=\id_{m\otimes m}\,,\\
	    \alpha_{m,b,m}=\bigoplus_{a\in A}\chi(a,b)^{-1}\cdot\id_{a}\,,\\
	    \alpha_{m,m,m}=\tau\cdot\sum_{a,b\in A}\chi(a,b)\cdot(\id_m\otimes[a]^\dagger)([b]\otimes\id_m)\,.
    \end{gather*}

    The reduction immediately preceding the proof establishes that this data is necessary to determine such a category.  Sufficiency then follows in a manner similar to the previous theorems.  The main subtlety lies in the equivalence classification, and the fact that the theorem makes no reference to $\tau$.

    Suppose there exists a monoidal equivalence $(F,J):\mathcal C_{\overline{\mathbb C}}(A,\chi,\tau)\to\mathcal C_{\overline{\mathbb C}}(A',\chi',\tau')$.  As we have seen before, $f:A\to A'$ will be an isomorphism, and $(-)^h\in\text{Gal}(\mathbb C/\mathbb R)$.  With our coefficients reduced as they are, Equations \ref{FullCBasis1}-\ref{FullCBasis8} simplify to the following.

    \begin{align}
        \chi'\big(f(a),f(b)\big)&=\chi(a,b)^h\,,\label{FullCBasis'3}\\
        \theta&=\delta\varphi\,,\label{FullCBasis'4}\\
        1&=\frac{\omega(a^{-1}b)\theta(a,a^{-1}b)}{\varphi(a)\omega(b)}\,,\label{FullCBasis'5}\\
        {\psi(a)}&={\overline{\varphi}(a)}\,,\label{FullCBasis'6}\\
        1&=\frac{\overline{\psi}(a)\omega(b)}{\theta(ba^{-1},a)\omega(ba^{-1})}\,,\label{FullCBasis'7}\\
        \tau'&=\tau\cdot\frac{\psi(a)\overline{\omega}(a)}{\varphi(b)\omega(b)}\,.\label{FullCBasis'8}
    \end{align}

    Since we are only checking the existence of an equivalence, we are free to normalize our equivalences by a monoidal isomorphism.  If $\mu:(F,J)\Rightarrow(F',J')$ is a monoidal isomorphism, then its components satisfy the following relations.

    \begin{align}
        \mu_{ab}\cdot\theta'(a,b)&=\theta(a,b)\cdot\mu_a\cdot\mu_b\label{MonEqFullC1}\\
        \varphi'(a)&=\varphi(a)\cdot\mu_a\label{MonEqFullC2}\\
        \psi'(a)&=\psi(a)\cdot\overline{\mu_a}\label{MonEqFullC3}\\
        \mu_{a}\cdot\omega'(a)&=\omega(a)\cdot\mu_m\cdot\overline{\mu_m}\label{MonEqFullC4}
    \end{align}
    By setting $\mu_a=\varphi(a)^{-1}$, we can completely trivialize $\varphi$, and hence also $\psi$ and $\theta$.  Having done this, Equations \ref{FullCBasis'4} and \ref{FullCBasis'7} show that $\omega$ is constant, so let us set that constant value to be $\omega(1)=\lambda$.  Equation \ref{MonEqFullC4} shows that we can use $\mu_m$ to assume that $|\lambda|=1$.

    The relations set out in the hypotheses of the theorem provide Equation \ref{FullCBasis'3}, so the only remaining nontrivial equation is Equation \ref{FullCBasis'8}, which becomes
    \begin{gather}
        \tau'\cdot\lambda^2\;=\;\tau\,.
    \end{gather}
    Since $\tau$ and $\tau'$ can only differ by a sign, we find that $\lambda=i^n$ for some $n$.
    
    Since the choice of $\lambda$ does not have any effect on the validity of Equations \ref{FullCBasis'3}-\ref{FullCBasis'7}, $\lambda=1$ allows for an equivalence when $\tau=\tau'$, and $\lambda=i$ allows for an equivalence when $\tau=-\tau'$.  It follows that the sign of $\tau$ does not control the existence of an equivalence at all.
    
    Thus the sign of $\tau$ is not an invariant of the category $\mathcal C_{\overline{\mathbb C}}(A,\chi,\tau)$ at all.  Knowing this, we can define $\mathcal C_{\overline{\mathbb C}}(A,\chi):=\mathcal C_{\overline{\mathbb C}}(A,\chi,\tfrac{1}{\sqrt{|A|}})$ to complete the theorem.
    
\end{proof}

\begin{remark}
The reader may recognize this as a skew-symmetric analogue of the classical Tambara-Yamagami classification.  Nondegenerate skew-symmetric bicharacters on finite abelian groups were classified in \cite{wallQuadraticForms}.  The classification involves the familiar `hyperbolic' bicharacters on the $p$-primary summands, with extra possibilities for the case when $p=2$.
\end{remark}

\begin{example}
    Let $A=\mathbb Z/2\mathbb Z=\langle t\rangle$.  The bicharacter $\chi(t,t)=-1$ is nondegenerate, and simultaneously symmetric and skew-symmetric.  Thus, in addition to giving rise to classical split Tambara-Yamagami categories, it can also be used to produce $\mathcal C_{\overline{\mathbb C}}(A,\chi)$.  The only difference between the two constructions is the Galois nontriviality of $m$.
\end{example}

\begin{example}
    Let $A=(\mathbb Z/4\mathbb Z)^2=\langle x,y\rangle$.  Define a skew-symmetric bicharacter by the formulas
    \[\chi(x,x)=1\,,\qquad\chi(x,y)=i\,,\qquad \chi(y,y)=-1\,.\]
    Since $\chi$ is nondegenerate, 
    we can construct $\mathcal C_{\overline{\mathbb C}}(A,\chi)$. This type of bicharacter has no nondegenerate analogues for odd $p$-primary groups.
\end{example}

\begin{example}\label{Eg:FullCG=1}
    Let $A=1$ be the trivial group.  There is only one bicharacter, and it is automatically nondegenerate.  In this case, the category $\mathcal C_{\overline{\mathbb C}}(A,\chi)$ is equivalent to the category $(\mathbb C,\mathbb C)$-bim from Example \ref{Eg:CCBim}.

    This category is pointed, and thus the equivalence classes of monoidal structures that it can have are classified by the twisted cohomology group $H^3\big(\text{Gal}(\mathbb C/\mathbb R)\,;\,\mathbb C^\times\big)=1$.  The fact that this group is trivial can be interpreted as another proof that the sign of $\tau$ is irrelevant, at least when $A$ is trivial.
\end{example}

\begin{proposition}\label{Prop:ComplexGaloisCatsAreRigid}
    The categories $\mathcal C_{\overline{\mathbb C}}(A,\chi)$ are rigid.
\end{proposition}

\begin{proof}
    As in the previous cases, we set $\text{ev}_m:=[1]$ and $\text{coev}_{m}:=\sqrt{|A|}\,[1]^\dagger$.  Upon observing that $\alpha_{m,m,m}^{-1}=\alpha_{m,m,m}$, the duality equations hold by a direct computation.
\end{proof}

\printbibliography

@article{etingofFusionCategories2005,
	title = {On fusion categories},
	volume = {162},
	issn = {0003-486X},
	url = {http://annals.math.princeton.edu/2005/162-2/p01},
	doi = {10.4007/annals.2005.162.581},
	language = {en},
	number = {2},
	urldate = {2021-10-14},
	journal = {Annals of Mathematics},
	author = {Etingof, Pavel and Nikshych, Dmitri and Ostrik, Victor},
	month = sep,
	year = {2005},
	pages = {581--642},
}

@article {etingofFusionCategoriesHomotopy2009,
    AUTHOR = {Etingof, Pavel and Nikshych, Dmitri and Ostrik, Victor},
     TITLE = {Fusion categories and homotopy theory},
      NOTE = {With an appendix by Ehud Meir},
   JOURNAL = {Quantum Topol.},
  FJOURNAL = {Quantum Topology},
    VOLUME = {1},
      YEAR = {2010},
    NUMBER = {3},
     PAGES = {209--273},
      ISSN = {1663-487X},
   MRCLASS = {18D10 (55S35)},
  MRNUMBER = {2677836},
MRREVIEWER = {Juan Mart\'{\i}n Mombelli},
       DOI = {10.4171/QT/6},
       URL = {https://doi.org/10.4171/QT/6},
}

@book{etingof2015tensor,
  title={Tensor Categories},
  author={Etingof, Pavel and Gelaki, Shlomo and Nikshych, Dmitri and Ostrik, Victor},
  isbn={9781470420246},
  lccn={2015006773},
  series={Mathematical Surveys and Monographs},
  url={https://books.google.com/books?id=NwM-CgAAQBAJ},
  year={2015},
  publisher={American Mathematical Society}
}

@article{TAMBARA1998692,
	title = {Tensor categories with fusion rules of self-duality for finite abelian groups},
	volume = {209},
	issn = {0021-8693},
	url = {https://www.sciencedirect.com/science/article/pii/S0021869398975585},
	doi = {https://doi.org/10.1006/jabr.1998.7558},
	number = {2},
	journal = {J. Algebra},
	author = {Tambara, Daisuke and Yamagami, Shigeru},
	year = {1998},
	pages = {692--707},
}

@article {theoSpinStatistics,
    AUTHOR = {Johnson-Freyd, Theo},
     TITLE = {Spin, statistics, orientations, unitarity},
   JOURNAL = {Algebr. Geom. Topol.},
  FJOURNAL = {Algebraic \& Geometric Topology},
    VOLUME = {17},
      YEAR = {2017},
    NUMBER = {2},
     PAGES = {917--956},
      ISSN = {1472-2747},
   MRCLASS = {57R56 (14A22 81T50)},
  MRNUMBER = {3623677},
MRREVIEWER = {Jorge Andres Devoto},
       DOI = {10.2140/agt.2017.17.917},
       URL = {https://doi.org/10.2140/agt.2017.17.917},
}

@article {etingofDescentAndForms,
    AUTHOR = {Etingof, Pavel and Gelaki, Shlomo},
     TITLE = {Descent and forms of tensor categories},
   JOURNAL = {Int. Math. Res. Not. IMRN},
  FJOURNAL = {International Mathematics Research Notices. IMRN},
      YEAR = {2012},
    NUMBER = {13},
     PAGES = {3040--3063},
      ISSN = {1073-7928},
   MRCLASS = {18D10},
  MRNUMBER = {2946231},
MRREVIEWER = {Gongxiang Liu},
       DOI = {10.1093/imrn/rnr119},
       URL = {https://doi.org/10.1093/imrn/rnr119},
}

@article {wallQuadraticForms,
    AUTHOR = {Wall, C. T. C.},
    TITLE = {Quadratic forms on finite groups, and related topics},
    JOURNAL = {Topol.},
    fJOURNAL = {Topology},
    VOLUME = {2},
    YEAR = {1963},
    PAGES = {281--298},
    ISSN = {0040-9383},
    MRCLASS = {10.16 (20.30)},
    MRNUMBER = {156890},
    MRREVIEWER = {T. Ono},
    DOI = {10.1016/0040-9383(63)90012-0},
    URL = {https://doi.org/10.1016/0040-9383(63)90012-0},
}

@article{sanfordThesis,
author={Sanford,Sean C.},
year={2022},
title={Fusion categories over non-algebraically closed fields},
journal={ProQuest Dissertations and Theses},
pages={1--172},
isbn={9798841750581},
url={https://proxyiub.uits.iu.edu/login?qurl=https\%3A\%2F\%2Fwww.proquest.com\%2Fdissertations-theses\%2Ffusion-categories-over-non-algebraically-closed\%2Fdocview\%2F2715420938\%2Fse-2},
}

@article{Frobenius1877,
author = {Frobenius, Georg},
journal = {J. für die Reine und Angew. Math.},
pages = {1-63},
title = {Ueber lineare Substitutionen und bilineare Formen.},
url = {http://eudml.org/doc/148343},
volume = {84},
year = {1877},
}

@unpublished{sinhGrCats,
author={Hoàng Xuân Sính},
year={1975},
title={Gr-Catégories},
note={(Unpublished)},
language={French},
url={https://agrothendieck.github.io/divers/GCS.pdf},
}

@article {bondersonOnInvariants,
    AUTHOR = {Bonderson, Parsa and Delaney, Colleen and Galindo, C\'{e}sar and
              Rowell, Eric C. and Tran, Alan and Wang, Zhenghan},
     TITLE = {On invariants of modular categories beyond modular data},
   JOURNAL = {J. Pure Appl. Algebra},
  FJOURNAL = {Journal of Pure and Applied Algebra},
    VOLUME = {223},
      YEAR = {2019},
    NUMBER = {9},
     PAGES = {4065--4088},
      ISSN = {0022-4049},
   MRCLASS = {18D10 (16T05 19D23)},
  MRNUMBER = {3944465},
MRREVIEWER = {Costel Gabriel Bontea},
       DOI = {10.1016/j.jpaa.2018.12.017},
       URL = {https://doi.org/10.1016/j.jpaa.2018.12.017},
}

@article {baezHigherDimAlgV,
    AUTHOR = {Baez, John C. and Lauda, Aaron D.},
     TITLE = {Higher-dimensional algebra. {V}. 2-groups},
   JOURNAL = {Theory Appl. Categ.},
  FJOURNAL = {Theory and Applications of Categories},
    VOLUME = {12},
      YEAR = {2004},
     PAGES = {423--491},
   MRCLASS = {18D05 (18D10 20J06)},
  MRNUMBER = {2068521},
MRREVIEWER = {R. H. Street},
}

@unpublished{zhuPrimes,
	title = {Pointed fusion categories over non-algebraically closed fields},
	author = {Plavnik, Julia and Sanford, Sean and Zhu, Sophie},
	note = {(In preparation)},
}
\end{document}